\documentclass[dvips,bj,preprint]{imsart}

\usepackage{amsthm,amsmath,amsfonts}
\usepackage{mathabx}
\usepackage[mathscr]{eucal}
\usepackage{enumerate,color,soul}
\RequirePackage[colorlinks,citecolor=blue,urlcolor=blue]{hyperref}

\startlocaldefs

\newtheorem{theorem}{Theorem}
\newtheorem{corollary}[theorem]{Corollary}

\newtheorem{lemma}[theorem]{Lemma}
\newtheorem{proposition}[theorem]{Proposition}
\newtheorem{remark}[theorem]{Remark}

\def\bds{\begin{displaystyle}}
\def\eds{\end{displaystyle}}
\renewcommand{\P}{{\mathbb{P}}}
\newcommand{\E}{{\mathbb{E}}}
\newcommand{\D}{{\mathbb{D}}}
\newcommand{\R}{{\mathbb{R}}}
\newcommand{\N}{{\mathbb{N}}}
\newcommand{\T}{{\mathbb{T}}}
\newcommand{\q}{{\quad}}

\endlocaldefs

\begin{document}

\begin{frontmatter}

\title{Convergence of the age structure of general schemes of population processes}
\runtitle{Convergence of Age Structure Processes} 

\begin{aug}



\author{\fnms{Jie Yen} \snm{Fan}\thanksref{a}\ead[label=e1]{jieyen.fan@monash.edu}},
\author{\fnms{Kais} \snm{Hamza*}\thanksref{a}\corref{}\ead[label=e2]{kais.hamza@monash.edu}},
\author{\fnms{Peter} \snm{Jagers}\thanksref{b}\ead[label=e3]{jagers@chalmers.se}}
\and
\author{\fnms{Fima} \snm{Klebaner}\thanksref{a}\ead[label=e4]{fima.klebaner@monash.edu}}

\runauthor{Fan, Hamza, Jagers and Klebaner}

\affiliation{Monash University\thanksmark{a} and
Chalmers University of Technology and University of Gothenburg\thanksmark{b}}

\address[a]{School of Mathematical Sciences, Monash University, Clayton, VIC 3800, Australia.
\printead{e1,e2,e4}}

\address[b]{Department of Mathematical Sciences, Chalmers University of Technology and University of Gothenburg, SE-412 96 Gothenburg, Sweden.
\printead{e3}}

\end{aug}

\begin{abstract} 
We consider a family of general branching processes with reproduction parameters  depending on the age of the individual as well as
the population age structure and a parameter $K$, which may  represent the carrying capacity. These processes are Markovian in the age structure.
In a previous paper \cite{HamJagKle13} the Law of Large Numbers as $K\to\infty$ was derived.
Here we prove the Central Limit Theorem, namely the weak convergence of the fluctuation processes in an appropriate Skorokhod space.
We also show that the limit is driven by a stochastic partial differential equation.
\end{abstract}



\end{frontmatter}


\section{Introduction}

A branching process is used to model a system of particles where each particle has a random lifespan
and gives birth to a random number of offspring at some point during lifetime or at death.
Classical frameworks of branching process include the Galton-Watson process in discrete time
and the Bellman-Harris branching process in continuous time.
In the Bellman-Harris framework, particles, independently of each other and with the same law,
live for a random length of time and reproduce at death a random number of offspring.
In this paper, we consider a much more general framework
introduced by Jagers and Klebaner (\cite{JagKle00}, \cite{JagKle11}).

Consider a population of size $z$ with ages $(a_1,a_2,\ldots,a_z)$.
This age structure can be represented by the measure $A=\sum_{i=1}^z \delta_{a_i}$ on ${\mathcal B}$, the Borel $\sigma$-field of $\R_+$,
where $\delta_a$ denotes the Dirac measure at $a$.
In particular, for a measurable set $B$, $A(B)$ represents the number of individuals with ages in $B$.
While the size of the population at time $t$ in the Bellman-Harris process is not Markov, the measure-valued process of ages  is.
The Markov property remains even when the life span and reproduction of individuals are allowed to depend on the whole population.

We allow reproduction and death to depend on not only the individual's age and the size of the population, but also the entire age structure of the population. 
In particular, as given in the examples in Section \ref{S:Example}, 
the reproduction and death could depend on the age, the population size, 
as well as other demographic features, through a so-called demographic kernel.
We allow also the parameters to depend on some parameter $K$, which could play the role of the carrying capacity of the habitat (\cite{JagKle11}).
Multiple offspring during life and at death is possible, to have a rather general model.
We are interested in the approximations when $K$ is large.

Similar questions have been answered in \cite{Kle93} and \cite{KleNer94} under the Galton-Watson setting,
where the reproduction of each particle depends on the carrying capacity,
but is otherwise independent and identically distributed conditionally
on the carrying capacity and the size of the population.
Oelschl\"ager \cite{Oel90} also answered a similar question in the context of birth-death processes,
deriving a Law of Large Numbers (LLN) and a Central Limit Theorem (CLT) for the empirical processes
of age-structured populations as the population size tends to infinity.

Tran  \cite{Tra06} (also \cite{Tra08} and \cite{FerTra09})
obtained a LLN and a CLT for a population model structured by traits and ages (not just  the physical age).
He generalises the standard model by including the
possibility of trait mutations and interactions (through a kernel) among individuals,
while keeping the dependence of the reproduction on
just the state (traits and ages) of that individual.
In contrast, we allow the births and deaths to depend on the age structure of the whole population.
Kaspi and Ramanan (\cite{KasRam11} and \cite{KasRam13}) obtained LLN and CLT for measure-valued queuing processes,
which inspired   this paper.

Convergence of measure-valued processes has been studied in various settings over the last decades.
This has been done also in the context of population or particle systems,
either giving results of the type of LLN only (e.g. \cite{Bor90}, \cite{MelTra12}, \cite{Met87}),
or together with CLT(e.g. \cite{BanEtal11}, \cite{Li11}, \cite{Mel98}, \cite{Oel90}, \cite{Tra06}).

The LLN for our model, given in \cite{HamJagKle13}, shows that under suitable assumptions on the parameters, the sequence of measure-valued processes $\bar A^K = A^K/K$ converges as $K\to \infty$ to a deterministic process $\bar A$ in the Skorokhod space $\D(\R_+, {\mathcal M}_+(\R_+))$, where ${\mathcal M}_+(\R_+)$ is the space of finite positive measures on $\R_+$, with its weak topology.
The limiting process is identified as the weak form of a generalised McKendrick-von Foerster Equation.
In this paper, we establish the CLT (see Theorem \ref{T:CLT}) for the age structure,
that is, the convergence of $Z^K = \sqrt{K}(\bar A^K- \bar A)$ in an appropriate space, and identify the limit.
In the limit (CLT), Fr\'echet derivatives of the rate functions naturally appear. They replace the ordinary derivatives in the density-dependent case where dependence is on the total mass of the measure.
Our CLT yields new results even in the classical case of constant parameters.

As usual, to establish convergence we show tightness and uniqueness of the limit.
The tightness is proved by using the Sobolev embedding approach and Aldous-Rebolledo tightness criteria, 
the method used in Bansaye et. al. \cite{BanEtal11}, Meleard \cite{Mel98}, and Tran \cite{Tra06}.
Since $Z^K$ is a signed measure-valued process, and the space of signed measures with the topology of weak convergence is not metrizable (\cite{BanEtal11}, \cite{Mel98}, \cite{Var58}), 
we embed the space of signed measures in suitable Sobolev spaces
(which are also Hilbert spaces),
and apply Sobolev embedding techniques with some Hilbertian properties.

While the Sobolev embedding technique has been much used (e.g. \cite{BanEtal11}, \cite{Bor90}, \cite{Mel98})  since being introduced by  Metivier \cite{Met87},  and  there
are seminal papers in the field  such as \cite{BanEtal11} and \cite{KasRam13}, our approach
has a number of differences.
We set up evolution equations for a branching process, fusing branching and stochastic analysis. This  is done by using the Ulam-Harris representation.
A simplifying technical feature of our model is that we can work on the bounded domain $\mathbb{T}^*:=[0,T+a^*]$, where $a^*$ is the age of the oldest individual alive at time 0 and we consider a finite time horizon $\T=[0,T]$. (Thus, $T+a^*$ is an upper bound to the age of the oldest individual alive at time $T$.) This boundedness of domain avoids the use of weighted Sobolev spaces (see page \pageref{Tstar}).

Section \ref{S:Model} sets up the model and gives a semimartingale representation to the process, 
with the proofs of some details postponed till Section \ref{S:SemiMGProofs}. 
Main results are stated in Section \ref{S:Results}, with the proof of the CLT in Section \ref{S:Proof} and the proofs of further results in Section \ref{S:ZInftyProofs}. Section \ref{S:Example} ends the paper with some examples.

Throughout this paper, we use $c$ with and without subscript
to denote constants that may be different from line to line, but all 
independent of $K$. $\N$ stands for the set of natural numbers   and
$\N_0$ for the set of non-negative integers.
For a Borel (positive or signed) measure $\mu$ on $E$ and a measurable function $f$ on $E$,
we write $(f,\mu) \equiv \int_E f(x) \mu(dx)$. 
The Skorokhod space $\mathbb{D}(\mathbb{T},\mathcal{M})$ consists of all c\`adl\`ag functions from $\mathbb{T}$ to $\mathcal{M}$.
We will take $\mathcal{M}$ to be a space of measures (for LLN)
and the dual of a suitable Sobolev space of functions (for CLT).

\section{Evolution equation and semimartingale decomposition} \label{S:Model}

In this section we set up the model and derive a semimartingale decomposition of the branching model, but leave the technical proofs to Section \ref{S:SemiMGProofs}.

We shall adopt the classical, well-known in branching (e.g. \cite{Har63}), Ulam-Harris labelling, as presented in \cite{Jag75} and developed in \cite{Jag89}.
We use the set
$$I = \bigcup_{n=1}^\infty \N^n$$
to denote all possible individuals;
$\N$ corresponds to the possible individuals of the starting generation,
$\N^2$ corresponds to the possible individuals of the second generation, and so forth.
We allow an arbitrary finite number of individuals at the start of the process at time $t=0$.
The individuals in the first (starting) generation are labelled $1,2,3,\dots$.
For each individual $y\in I$, the children of $y$ are consecutively labelled $y1, y2, y3, \dots$ as they are born.
Here $yi$ is the concatenated vector of the coordinates of $y\in I$ and $i\in \N$.

We assume that the age of each individual increases at rate 1 until the individual dies.
Upon death it may split into a random   number of offspring.
During its lifetime the individual may  give birth to a random  number of offspring.
The offspring generated in both situations are referred to as the children of the individual, and both situations are considered as births.

We denote by $\tau_y$, $\lambda_y$ and $\sigma_y = \tau_y+\lambda_y$ respectively the time of birth, the life span and the death time of individual $y$.
In particular, the maternal age for the birth of the $j$th child (during lifetime or by splitting at death) of individual $y$ is $\tau_{yj} - \tau_y$.
Also, if $y$ has precisely $n$ children, then
$$\tau_y < \tau_{y1} \le\ldots\le \tau_{yn} \le \sigma_y\mbox{ and }\tau_{y(n+1)} = \infty.$$

The population starts from an initial age distribution $A_0$ with mass one at given ages $x_1, x_2, \dots, x_{(1,A_0)}$
and the population size $(1,A_0)$ is assumed to be finite.
Put $\tau_i = -x_i$, $i=1,2,\dots,(1,A_0)$ for the birth times of these ancestors (first generation).

The age distribution $A_t$ at time $t\ge0$ allots a unit weight to the age ($t-\tau_y$) of each individual ($y\in I$) that is alive at time $t$,
\begin{equation}\label{aget}
A_t = \sum_{y\in I} \mathbf{1}_{\tau_y\le t<\sigma_y}\delta_{t-\tau_y} .
\end{equation}
For each $t$, $A_t$ is a finite discrete measure on $\R_+$, in particular, $A_t \in \mathcal{M}_+(\R_+)$,
and the collection $(A_t)_{t\ge0}$ is known as the age structure process of the population.

Two processes that determine the evolution of population are the way the individuals enter and the way they exit.
Denote by  $B(t)$ the number of individuals born by time $t$, and by $D(x,t)$ the number of individuals who died by time $t$ and whose life span was $x$ or less, then
$$B(t) = \sum_{y\in I} \mathbf{1}_{\tau_y\le t},\q D(x,t) = \sum_{y\in I} \mathbf{1}_{\lambda_y\le x, \sigma_y\le t}.$$
Before we give the fundamental equation for the   evolution of the population, we make an important observation (which allows us to work on a bounded time interval and to avoid using weighted Sobolev spaces).

Recall that $a^*$ is the age of the oldest individual in the starting generation, that is,
$$a^* = \inf\{x>0: A_0((x,\infty))=0\}.$$
Since we look at the convergence on a finite time interval $[0,T]$,
the age of any individual at time $t\le T$ will not be more than $T+a^*$,
thus the support of $A_t$ is contained in $[0,T+a^*]$.
Henceforth denote by $T^*=T+a^*$. \label{Tstar}

While our focus is indeed on functions of a single variable, the proof of the CLT requires a semimartingale decomposition for functions of two variables.
Consequently, we consider test functions of two variables $f(x,t)$ whose domain is limited to the bounded rectangle $\mathbb{T}^*\times\mathbb{T}$, where $\T^*=[0,T^*]$ is the age space and $\T=[0,T]$ is the time space. 
In what follows, we will also write $f_t(x)$ to mean $f(x,t)$ and use the two notations interchangeably.

We have the following basic equation, with proof in Section \ref{S:SemiMGProofs}.

\begin{proposition} \label{P:BasicEqnInfty}
For any $f\in C^{1,1}(\mathbb{T}^*\times\mathbb{T})$ and $t\in \mathbb{T}$,
the age structure process $A$ satisfies
\begin{multline}  \label{Fund1}
(f_t,A_t) = (f_0,A_0)
+ \int_0^t ( \partial_1f_s+\partial_2f_s,A_s ) ds \\
+\int_{[0,t]} f(0,s)B(ds) - \int_{\mathbb{T}^*\times[0,t]} f(x,s)D(dx,ds).
\end{multline}
\end{proposition}

To arrive at  compensators for the two processes in the RHS of \eqref{Fund1}, we assume the existence
of birth and death rates,   dependent on the age and also upon the
population age structure (cf. \cite{JagKle00}).
The number of births by time $t$ consists of births by living
mothers and births by splitting, $B= \widecheck B + \widehat B$. An individual  aged $x$ at time $t$ gives birth at rate $b_{A_t}(x)$ and dies at rate $h_{A_t}(x)$, allowing for multiple births.

Denote the random variables $\widecheck \xi_{A_t}(x)$ and $\widehat \xi_{A_t}(x)$ the number of children at a bearing of a living individual aged $x$ at time $t$ and at splitting (i.e. death), respectively. 
Let $\widecheck m_{A_t}(x) = \E[\widecheck \xi_{A_t}(x)| A_t]$ and  $\widehat m_{A_t}(x) = \E[\widehat \xi_{A_t}(x)| A_t]$. 
Thus the mean intensity of births of an individual aged $x$ at time $t$ is
$\widecheck m_{A_t}(x)b_{A_t}(x)+\widehat m_{A_t}(x)h_{A_t}(x)$.
We also denote the conditional second moment of the
number of children at a bearing of a living individual aged $x$ at time $t$ by
$\widecheck v_{A_t}(x)$, and similarly the conditional second moment of the number of children at
splitting by $\widehat v_{A_t}(x)$.

The compensators of the birth and death terms in \eqref{Fund1} are given by the following results, with proof in Section \ref{S:SemiMGProofs}.

\begin{proposition} \label{P:Compensators}
For every $f\in C(\mathbb{T}^*\times\mathbb{T})$ and $t\in \mathbb{T}$,
\begin{gather*}
\int_0^t(f_sh_{A_s},A_s)ds, 
\q \int_0^t f_s(0) (b_{A_s} \widecheck m_{A_s}, A_s) ds,
\q \mbox{and}\q \int_0^t f_s(0) (h_{A_s} \widehat m_{A_s}, A_s) ds
\end{gather*}
are the compensators of 
\begin{gather*}
\int_{\mathbb{T}^*\times[0,t]} f(x,s)D(dx,ds), 
\q \int_{[0,t]}f(0,s)\widecheck B(ds), 
\q\mbox{and}\q \int_{[0,t]}f(0,s)\widehat B(ds)
\end{gather*}
respectively.
\end{proposition}

Having found the compensators we identify the relevant martingales.
The proof of the following proposition is standard and is therefore omitted.

\begin{proposition} 
The following processes are martingales
\begin{align*}
M_{D,f}(t) &:= \int_{\mathbb{T}^*\times[0,t]} f(x,s)D(dx,ds) - \int_0^t (f_sh_{A_s},A_s) ds \\
M_{\widecheck B,f}(t) &:= \int_{[0,t]} f(0,s) \widecheck B(ds) - \int_0^t f_s(0) (b_{A_s} \widecheck m_{A_s}, A_s) ds \\
M_{\widehat B,f}(t) &:= \int_{[0,t]} f(0,s) \widehat B(ds) - \int_0^t f_s(0) (h_{A_s} \widehat m_{A_s}, A_s) ds
\end{align*}
with predictable quadratic variations
\begin{gather*}
\big<M_{D,f}\big>_t = \int_0^t (f^2_sh_{A_s}, A_s) ds, \q
\big<M_{\widecheck B,f}\big>_t  = \int_0^t f^2_s(0) (b_{A_s}\widecheck v_{A_s}, A_s) ds,\\
\big<M_{\widehat B,f}\big>_t = \int_0^t f^2_s(0) (h_{A_s}\widehat v_{A_s}, A_s) ds.
\end{gather*}
\end{proposition}

We combine the rates $n= b\widecheck m+h\widehat m$ and $w= b\widecheck v+h\widehat v$, and
also the martingales. From the basic equation \eqref{Fund1} we obtain the following semimartingale decomposition, with proof in Section \ref{S:SemiMGProofs}.

\begin{proposition} \label{P:SemiMGt}
For $t\in\mathbb{T}$ and $f\in C^{1,1}(\mathbb{T}^*\times\mathbb{T})$,
\begin{equation} \label{E:BasicEqnM2}
(f_t,A_t) = (f_0,A_0) + \int_0^t \big( L_{A_s}f_s ,A_s \big) ds
+ M^f_t,
\end{equation}
where
\begin{equation*}
L_{A}f(x,s) = \partial_1f(x,s) + \partial_2f(x,s) - f(x,s)h_{A_s} + f(0,s)n_{A_s}
\end{equation*}
and $M_t^f$ is a locally-square-integrable martingale with predictable quadratic variation
\begin{equation}\label{QVMtf2}
\big<M^f\big>_t = \int_0^t \big( f^2_s(0) w_{A_s} + h_{A_s}f^2_s - 2f_s(0)h_{A_s}\widehat m_{A_s}f_s, A_s \big) ds.
\end{equation}
\end{proposition}

\begin{remark}
The predictable quadratic covariation of the martingale with two test functions can also be obtained.
For $f,g \in C(\mathbb{T}^*\times\mathbb{T})$ and $t\in \mathbb{T}$,
$$\big<M^f,M^g\big>_t = \int_0^t \Big( f_s(0)g_s(0) w_{A_s} + h_{A_s}f_sg_s - h_{A_s}\widehat m_{A_s} \big(f_s(0)g_s+g_s(0)f_s\big), A_s \Big) ds.$$
\end{remark}

In particular, taking $f$ as a function of the first variable $x$ only, we recover Equation (2.6) of \cite{JagKle00}, stated again here for completeness.

\begin{corollary}
For $t\in\mathbb{T}$ and $f \in C^1(\mathbb{T}^*)$,
\begin{equation}\label{E:Mtf}
(f,A_t) = (f,A_0) + \int_0^t (L_{A_s}f, A_s) ds + M_t^{f},
\end{equation}
where
\begin{equation*}
L_{A}f = f' - h_{A}f + f(0)n_{A}
\end{equation*}
and $M_t^f$ is a locally-square-integrable martingale with predictable quadratic variation
\begin{equation}\label{QVMf}
\big<M^f\big>_t = \int_0^t \big( f^2(0) w_{A_s} + h_{A_s}f^2 - 2f(0)h_{A_s}\widehat m_{A_s}f, A_s \big) ds.
\end{equation}
\end{corollary}

\section{A Central Limit Theorem} \label{S:Results}
We now look at the case of a branching process dependent on some (large) index $K$;
$K$ may, for example, represent the population carrying capacity,
a threshold below which the process is supercritical and above which it is subcritical.
The notion of carrying capacity plays a great role in biological population dynamics.
The interest is to approximate such a process for large $K$.
This leads to consider a family of branching processes indexed by $K$.
All objects introduced in the previous sections will now carry the extra label $K$:
$A^K$, $b^K$, $h^K$ etc.
The qualifiers $\widehat\ $ and $\widecheck\ $ (of $m$ and $v$) will be dropped in any statement that refers to either qualifier.

Throughout the remainder of the paper, we make one simplifying (and reasonable) assumption in that the ages of all individuals in all starting generations
are bounded. We denote (with a slight abuse of notation) by $a^*$ ``the age of the oldest individual'' at $t=0$:
$$a^* := \sup_{K\ge1}\big( \inf\{x>0: A_0^K ((x,\infty))=0\} \big) <\infty.$$
As before, $\mathbb{T}^*=[0,T^*]$ with $T^*=T+a^*$.
For each $K$, $A^K = (A^K_t)_{t\in\mathbb{T}}$ is a c\`adl\`ag positive measure-valued process on $\mathbb{T}^*$,
i.e. $(A^K_t)_{t\in\mathbb{T}} \in \mathbb{D}(\mathbb{T},\mathcal{M}_+(\mathbb{T}^*))$.
Without loss of generality, we assume that $A^K_0$ is deterministic.

As we shall focus on situations where $\bar A^K := A^K/K$ converges to a non-degenerate limit, a new parametrisation of the intensities is needed,
one that involves $\bar A^K$ rather than $A^K$ itself.

We have, immediately from Equation \eqref{E:Mtf}, the following evolution of $\bar{A}^K$:
\begin{equation}\label{E:AbarK}
(f,\bar A^K_t) = (f,\bar A^K_0) + \int_0^t (L^K_{\bar A^K_s}f, \bar A^K_s) ds + \frac{1}{{K}} M_t^{f,K},
\end{equation}
where
\begin{equation}\label{OpLAK}
L^K_{A}f = f' - h^K_{A}f + f(0)n^K_{A}
\end{equation}
and $M_t^{f,K}$ is a martingale. A similar representation with functions of two variables is also used later  in proofs.

\subsection{The Law of Large Numbers}  The LLN was    established in \cite{HamJagKle13} under
  the  following  conditions, referred to as  {\it smooth demography}:
\begin{enumerate}
\item[(C0)]
The model parameters $b$, $h$, $m$ and $v$ are uniformly bounded,
that is, $\sup_{K,A,x} b^K_A(x) <\infty$, et cetera.
Note that the supremum with respect to $A$ is taken over $A\in \mathcal{M}_+(\T^*)$.
\item[(C1)]
The model parameters $b$, $h$ and $m$ are normed uniformly Lipschitz in the following
sense: there is a $c>0$ such that for all $K\ge1$,
$||b^K_A-b^K_B||_\infty \le c||A-B||$,
where $||\mu|| := \sup_{||f||_\infty\le1, f \textnormal{ continuous}} |(f, \mu)|$;
the same applies to $h$ and $m$.
\item[(C2)] The limit (pointwise in $A$ and uniform in $x$) $\lim b^K_{A}=:b^\infty_{A}$ exists;
the same applies to limits $\lim h^K_{A}=:h^\infty_{A}$ and $\lim m^K_{A}=:m^\infty_{A}$.
\item[(C3)] $\bar A^K_0 \Rightarrow \bar A_0^\infty$, $\sup_K(1,\bar A_0^K)<\infty$.
\end{enumerate}

We remark that in \cite{HamJagKle13}, the Prokhorov metric is used for (C1).
However, since we shall work in spaces $C^{-j}$ and $W^{-j}$ (see Section \ref{S:Spaces}) for the CLT,
it is more natural to use the norms in these spaces.
In our context, the norm $||\cdot||$ coincides with $||\cdot||_{C^{-0}}$ defined in Section \ref{S:Spaces}.
It can be shown that the LLN remains valid with this (C1).

\begin{theorem} [\cite{HamJagKle13}]
Under the smooth demography condition, as $K\to\infty$, $\bar{A}^K$ converges weakly
in the Skorokhod space $\mathbb{D}(\mathbb{T}, \mathcal{M}_+(\mathbb{T}^*))$
to the limiting process $\bar{A}$, which is deterministic and satisfies, for $f\in C^1(\mathbb{T}^*)$ and $t\in\mathbb{T}$,
\begin{equation} \label{E:LLN}
(f, \bar{A}_t)
= (f, \bar{A}_0) + \int_0^t (L^\infty_{\bar A_s}f, \bar A_s)ds,
\end{equation}
where $L^\infty_{A}f = f' - h^\infty_{A}f + f(0)n^\infty_{A}$
and $n^\infty_A = b^\infty_A \widecheck m^\infty_A + h^\infty_A \widehat m^\infty_A$.
\end{theorem}
It follows by the Monotone Class Theorem (e.g.  \cite[I.22.1]{DelMey78}) that \eqref{E:LLN}   also  holds for
test functions of two variables, $f\in C^{1,1}(\mathbb{T}^*\times\mathbb{T})$, and $t\in\mathbb{T}$.
This fact will be used later in the representation of the fluctuation process.

As remarked in \cite{HamJagKle13},
if $\bar A_0$ has a density, then $\bar A_t$ has a density; call it $a(x,t)$. In such case,
Equation \eqref{E:LLN} is the weak form of the McKendrick-von Foerster equation for the density:
$$\Big(\frac{\partial}{\partial x}+\frac{\partial}{\partial t}\Big)a(x,t) = -a(x,t)h^\infty_{\bar A_t}(x),
\q a(0,t) = \int_0^{t+a^*}n^\infty_{\bar A_t}(x)a(x,t)dx.$$

\subsection{The fluctuation process $Z^K$}

For each $t$ and $K$, $Z_t^K := \sqrt{K} (\bar{A}_t^K - \bar{A}_t)$
is a finite signed measure that, in view of \eqref{E:AbarK} and \eqref{E:LLN}, can be represented as
\begin{equation}\label{E:fZK}
(f,Z_t^K) = (f,Z_0^K)
+ \sqrt{K} \int_0^t \left( L^K_{\bar A^K_s}f - L^\infty_{\bar A_s}f , \bar{A}_s \right) ds
+ \int_0^t \left( L^K_{\bar A^K_s}f , Z^K_s \right) ds + \tilde{M}_t^{f,K},
\end{equation}
where $\tilde{M}_t^{f,K} = M_t^{f,K}/\sqrt{K}$ is a martingale with predictable quadratic variation
\begin{equation*}
\big<\tilde M^{f,K}\big>_t = \int_0^t \big( f^2(0) w^K_{\bar A^K_s} + h^K_{\bar A^K_s}f^2 - 2f(0)h^K_{\bar A^K_s}\widehat m^K_{\bar A^K_s}f, \bar A^K_s \big) ds.
\end{equation*}

\subsection{Relevant spaces and embeddings} \label{S:Spaces}
Let $C^j(\mathbb{T}^*)$, $j\in\N_0$, denote the space of continuous functions on $\mathbb{T}^*$ with continuous derivatives up to order $j$.
Since $\mathbb{T}^*$ is a bounded domain, the functions in $C^j(\mathbb{T}^*)$ as well as their $j$ derivatives are bounded with the norm
$$||f||_{C^j(\mathbb{T}^*)} = \max_{0\le i\le j}\; \sup_{x\in \mathbb{T}^*} |f^{(i)}(x)|.$$
The  Sobolev space $W^j(\mathbb{T}^*)$ is the closure of $C^\infty(\mathbb{T}^*)$ with respect to the norm
$$||f||_{W^j(\mathbb{T}^*)} = \bigg(\sum_{i=0}^j \int_{\mathbb{T}^*} \big(f^{(i)}(x)\big)^2 dx \bigg)^{1/2},$$
where $f^{(i)}$ is the (weak) derivative of $f$ (see e.g. \cite{AdaFou03}).
The space $W^j(\mathbb{T}^*)$ is a Hilbert space with inner product
 $\left<f,g\right>_{W^j(\mathbb{T}^*)} = \sum_{i=0}^j \int_{\mathbb{T}^*} f^{(i)}(x) g^{(i)}(x) dx.$

For the rest of this paper, we assume, unless otherwise specified, that functions are defined on the domain $\mathbb{T}^*$
and suppress the label $\mathbb{T}^*$; e.g. $W^j$ means $W^j(\mathbb{T}^*)$.

The following embeddings hold:
$$C^j \hookrightarrow W^j \ ,\ \
W^{j+1} \hookrightarrow C^j \q \textnormal{and} \q
W^{j+1} \underset{H.S.}{\hookrightarrow} W^j,$$
where H.S. stands for Hilbert-Schmidt embedding.
Let $C^{-j}(\mathbb{T}^*)$ and $W^{-j}(\mathbb{T}^*)$ denote
the dual spaces of, respectively $C^{j}(\mathbb{T}^*)$ and $W^{j}(\mathbb{T}^*)$.
Then,
$$W^{-j} \hookrightarrow C^{-j} \ ,\ \
C^{-j} \hookrightarrow W^{-(j+1)} \q \textnormal{and} \q
W^{-j} \underset{H.S.}{\hookrightarrow} W^{-(j+1)}.$$
In particular, we have
$$C^{-0} \hookrightarrow C^{-1} \hookrightarrow W^{-2} \underset{H.S.}{\hookrightarrow} W^{-3} \underset{H.S.}{\hookrightarrow} W^{-4}.$$

As a signed measure, $Z^K_t$ belongs to $C^{-0}$ for each $t$ and $K$.
To make use of representation \eqref{E:fZK}, we consider the process $Z^K$ as a process taking values in $C^{-1}$.
The technicality in establishing Aldous' tightness condition ((B) of Lemma \ref{L:AldousReb}) requires the embedding
$C^{-1} \hookrightarrow W^{-2} \hookrightarrow W^{-3} \hookrightarrow W^{-4}$.
In particular, with $C^{-1} \hookrightarrow W^{-2}$, the boundedness of $\E[||Z^K_t||_{W^{-2}}]$ is obtained (Proposition \ref{P:ZW-2Bound}), which is used to obtain the boundedness of $\E[ \sup_{t\le T} ||Z^K_t||_{W^{-3}}]$ (Proposition \ref{P:ZW-3Bound}), which is in turn used to establish the Aldous tightness criterion of $Z^K$ in $\mathbb{D}(\mathbb{T},W^{-4})$ (Proposition \ref{P:CheckT2}).
The Hilbert-Schmidt embedding $W^{-2} \underset{H.S.}{\hookrightarrow} W^{-4}$ is used to identify a compact set in order to establish coordinate tightness ((A) of Lemma \ref{L:AldousReb}).


We shall use the following general results, the proofs of which are standard and therefore omitted.
For any $f\in C^j$ and $g\in W^j$, $j\in \N_0$,
\begin{equation} \label{E:fgNorm}
||fg||_{W^j} \le c ||f||_{C^j} ||g||_{W^j}.
\end{equation}
Let $(p^j_l)_{l\ge1}$ denote a complete orthonormal basis of $W^j$,   $j\in \N$.
Then, for any $x_1\in\mathbb{T}^*$ and $x_2\in\mathbb{T}^*$,
\begin{equation} \label{E:CONB}
\Big|\sum_{l\ge1} p^j_l(x_1)p^j_l(x_2)\Big| \le c.
\end{equation}

\subsection{Statement of the Central Limit Theorem} \label{S:CLT}
Further to (C0)-(C3), we shall make the following assumptions.

\begin{enumerate}
\item[(A0)]
Conditions (C1) and (C2) hold also for $v$.
\item[(A1)]
$\Xi:= \sup_{x,A,K} \widecheck \xi^K_{A}(x) \vee \sup_{x,A,K} \widehat \xi^K_{A}(x)$ is in $L^2$.
\item[(A2)]
The reproduction parameters $b_A^K(x)$, $h^K_A(x)$ and $m^K_A(x)$ and their limits (in the sense of (C2)) are in $C^4$,
in the argument $x$, with convergence in $C^4$.
Moreover,
$\sqrt K \sup_A ||b^K_A-b^\infty_{A}||_\infty \rightarrow 0$ as $K\to\infty$,
$\sup_{K,A} ||b^K_A||_{C^3} <\infty$ and $\sup_A ||b^\infty_A||_{C^4} <\infty$;
similarly for parameters $h$ and $m$.
\item[(A3)]
The limiting parameters (as functions of $A$) are Fr\'echet differentiable at every $A$.
Namely, for every $A_0$,
there exists a continuous linear operator $\partial_A b^\infty_{A_0} : W^{-4} \to L_{\infty}$ such that
$$\lim_{||B||_{W^{-4}}\to 0} \frac{1}{||B||_{W^{-4}}} ||b^\infty_{A_0+B}-b^\infty_{A_0}-\partial_A b^\infty_{A_0}(B)||_\infty=0.$$
Moreover, $\sup_{A_0} ||\partial_A b^\infty_{A_0}||_{\mathbb{L}^{-4}} \le c$,
where $\mathbb{L}^{-4} = L(W^{-4},L_\infty)$ denotes  the space of continuous linear mappings from $W^{-4}$ to $L_\infty$.
The same applies to parameters $h^\infty$ and $m^\infty$.
\item[(A4)]
$Z^K_0$ converges to $Z^\infty_0$ in $W^{-4}$
and $\sup_K ||Z^K_0||_{W^{-2}} <\infty$.
\end{enumerate}

\begin{theorem} \label{T:CLT}
Assume  (A0)--(A4) in addition to the smooth demography condition (C0)--(C3).
Then, as $K\to \infty$, the process $(Z^K_t)_{t\in\mathbb{T}}$ converges weakly in $\mathbb{D}(\mathbb{T},W^{-4})$
to the process $(Z_t)_{t\in\mathbb{T}}$ that satisfies the equation, for $f \in W^4$,
\begin{multline} \label{E:fZInfty}
(f,Z_t) = (f,Z_0) +
\int_0^t (-\partial_A h^\infty_{\bar{A}_s}(Z_s) f + f(0)\partial_A n^\infty_{\bar{A}_s}(Z_s), \bar{A}_s) ds \\
+ \int_0^t (f' - h^\infty_{\bar{A}_s}f + f(0)n^\infty_{\bar{A}_s}, Z_s) ds + \tilde{M}^{f,\infty}_t
\end{multline}
where $n^\infty_A = b^\infty_A \widecheck m^\infty_A + h^\infty_A \widehat m^\infty_A$
and $\tilde{M}_t^{f,\infty}$ is a continuous Gaussian martingale with predictable quadratic variation
$$\big<\tilde{M}^{f,\infty}\big>_t = \int_0^t \left(f^2(0) w_{\bar{A}_s}^\infty + h_{\bar{A}_s}^\infty f^2 - 2f(0) h_{\bar{A}_s}^\infty \widehat m_{\bar{A}_s}^\infty f , \bar{A}_s\right) ds,$$
with $w_A^\infty = b_A^\infty \widecheck v_A^\infty + h_A^\infty \widehat v_A^\infty$.
\end{theorem}

\begin{corollary}[SPDE] \label{C:SPDE}
The limiting process $(Z_t)_{t\in\mathbb{T}}$ satisfies the following SPDE:
\begin{multline*}
dZ_t(dx) = -\partial_Ah^\infty_{\bar A_t}(Z_t)(x) \bar A_t(dx) dt + (\partial_An^\infty_{\bar A_t}(Z_t), \bar A_t) dt \delta_0(dx) \\
-(Z_t)'(dx)dt - h^\infty_{\bar A_t}(x) Z_t(dx) dt + (n^\infty_{\bar A_t}, Z_t) dt\delta_0(dx) + d\tilde{M}^{\infty}_t(dx),
\end{multline*}
where $\tilde{M}^\infty$ is a Gaussian martingale measure such that
$(f,\tilde{M}^\infty_t) = \tilde{M}^{f,\infty}_t$, and
  $(Z^\infty_t)'$ is defined by $(f,(Z^\infty_t)')=-(f',Z^\infty_t)$.
\end{corollary}

\begin{proposition} \label{P:EZtMeasure}
Suppose that $\partial_Ah^\infty_{A_0}(B)(x)$ has the form $\int g^h(A_0,x,y) B(dy)$ for some $g^h(A,x,\cdot) \in W^4$ with $\sup_{A,x} ||g^h(A,x,\cdot)||_{W^4} <\infty$,
and similarly $\partial_An^\infty_{A_0}(B)(x)$ is of the form $\int g^n(A_0,x,y) B(dy)$ for some $g^n(A,x,\cdot) \in W^4$ with $\sup_{A,x} ||g^n(A,x,\cdot)||_{W^4} <\infty$.
Then, $\nu_t:f\mapsto \E[(f,Z_t)]$ is a signed measure.
\end{proposition}

The proofs of Corollary \ref{C:SPDE} and Proposition \ref{P:EZtMeasure} is postponed to Section \ref{S:ZInftyProofs}.


\section{Proofs of Propositions 1, 2 and 4} \label{S:SemiMGProofs}

\begin{proof} [Proof of Proposition \ref{P:BasicEqnInfty}]
Note that $(f_t,A_t) = \sum_{y\in I}f(t-\tau_y,t)\mathbf{1}_{\tau_y\le t<\sigma_y}$.
Let $g(t) = f(t-\tau_y,t)$, then
$g'(t) = \partial_1f(t-\tau_y,t) + \partial_2 f(t-\tau_y,t)$
and
$$g(t)\mathbf{1}_{\tau_y\le t<\sigma_y} - g(0)\mathbf{1}_{\tau_y<0} = \int_0^t g'(s)\mathbf{1}_{\tau_y\le s<\sigma_y}ds
+ g(\tau_y)\mathbf{1}_{0\le \tau_y\le t} - g(\sigma_y)\mathbf{1}_{\sigma_y\le t}.$$
Summing over $y$, we get \eqref{Fund1}.
\end{proof}

\begin{proof} [Proof of Proposition \ref{P:Compensators}]
Let $H^f_t = \int_0^t(f_sh_{A_s},A_s)ds$ and $Q^f(t) = \int_{\mathbb{T}^*\times[0,t]} f(x,s)D(dx,ds)$. 
By the very definition of death rate, and the convention that all rates vanish for negative arguments,
$\mathbf{1}_{\sigma_y\le t} - \int_0^{t\wedge \sigma_y} h_{A_s}(s-\tau_y) ds$
is a martingale.
That is, for any bounded function $g$,
$$\E\big[ g(\sigma_y) \mathbf{1}_{\sigma_y>t} | \mathcal{F}_t \big] = \int_t^\infty g(s) \E\big[ h_{A_s}(s-\tau_y)\mathbf{1}_{\sigma_y>s} | \mathcal{F}_t \big] ds,$$
where $\mathcal{F}=\{\mathcal{F}_t\}$ is the natural filtration of the age structure process $A$. Equivalently, 
$$\P(\sigma_y\in ds\cap(t,+\infty)|\mathcal{F}_t) = \mathbf{1}_{(t,+\infty)}(s)\E\big[ h_{A_s}(s-\tau_y)\mathbf{1}_{\sigma_y>s} | \mathcal{F}_t\big]ds.$$
In particular
$$\lim_{\delta\downarrow0}\frac1\delta\P(t<\sigma_y\le t+\delta|\mathcal{F}_t) = h_{A_t}(t-\tau_y)\mathbf{1}_{\sigma_y>t}.$$
Now,
$Q^f_t = \sum_{y\in I} \mathbf{1}_{\sigma_y\le t} f(\lambda_y, \sigma_y)$
is adapted to the filtration $\mathcal{F}$ and for any $u>0$,
\begin{align*}
\E\big[Q^f_{t+u}\big|\mathcal{F}_t\big]
&= Q^f_t + \sum_{y\in I} \E\big[ f(\sigma_y-\tau_y,\sigma_y)\mathbf{1}_{t<\sigma_y\le t+u} | \mathcal{F}_t\big] \\
&= Q^f_t + \sum_{y\in I} \int_t^{t+u} \E\big[ f(s-\tau_y,s)h_{A_s}(s-\tau_y) \mathbf{1}_{\sigma_y>s} |\mathcal{F}_t\big] ds;
\end{align*}
and similarly, $H^f$ is adapted to $\mathcal{F}$, continuous and
$$\E\big[H^f_{t+u}\big|\mathcal{F}_t\big]
= H^f_t + \sum_{y\in I} \E\bigg[ \int_t^{t+u} f(s-\tau_y,s)h_{A_s}(s-\tau_y)\mathbf{1}_{\tau_y\le s<\sigma_y} ds \Big|\mathcal{F}_t\bigg].$$
Hence, $H^f$ is the compensator of $Q^f$, viewing that $h_{A_s}(s-\tau_y)=0$ for $s<\tau_y$.

The proof for other compensators follows from the fact that
$\int_0^t (b_{A_s} \widecheck m_{A_s},A_s)ds$  and $\int_0^t (h_{A_s} \widehat m_{A_s},A_s)ds$ are compensators for $\widecheck B$ and $\widehat B$.
\end{proof}

\begin{proof} [Proof of Proposition \ref{P:SemiMGt}]
It remains to prove \eqref{QVMtf2}. 
Note that $M^f_t = M_{\widecheck B,f}(t) + M_{\widehat B,f}(t) - M_{D,f}(t)$, and that the martingales $M_{D,f}$, $M_{\widecheck B,f}$ and $M_{\widehat B,f}$ are purely discontinuous.
Since $M_{\widecheck B,f}$ and $M_{\widehat B,f}$ do not jump together,
$\big[M_{\widecheck B,f}, M_{\widehat B,f}\big]_t$ and thus $\big<M_{\widecheck B,f}, M_{\widehat B,f}\big>_t$ are zero.
Similarly for $M_{D,f}$ and $M_{\widecheck B,f}$, giving $\big<M_{D,f}, M_{\widecheck B,f}\big>_t =0$.
However, $M_{D,f}$ and $M_{\widehat B,f}$ jump together when there is a birth by splitting with
$\Delta M_{D,f}(t) = \sum_{y\in I} f(\lambda_y,t) \mathbf{1}_{\sigma_y=t}$ and
$\Delta M_{\widehat B,f}(t) = \sum_{y\in I} f(0,t) \mathbf{1}_{\sigma_y=t} \sum_{i\in \N}\mathbf{1}_{\tau_{yi}=t}$.
Therefore,
$$\big[M_{D,f}, M_{\widehat B,f}\big]_t
= \sum_{s\le t} \sum_{y\in I} f(0,s) f(\lambda_y,s) \mathbf{1}_{\sigma_y=s} \sum_{i\in \N}\mathbf{1}_{\tau_{yi}=s}$$
and its compensator
$$\big<M_{D,f}, M_{\widehat B,f}\big>_t
= \int_0^t f_s(0) \big( f_sh_{A_s}\widehat m_{A_s}, A_s\big) ds.$$
Thus, $\big<M^f\big>_t = \big<M_{\widecheck B,f}\big>_t + \big<M_{\widehat B,f}\big>_t + \big<M_{D,f}\big>_t -2 \big<M_{D,f}, M_{\widehat B,f}\big>_t$, and we have \eqref{QVMtf2}.
\end{proof}


\section{Proof of the Central Limit Theorem} \label{S:Proof}

We establish the tightness of the sequence $Z^K$,
and show the uniqueness of the limit.

\subsection{Tightness of $Z^K$}

First we prove a result for the tightness of   $W^{-j}$-valued processes in the Skorokhod space $\mathbb{D}(\mathbb{T},W^{-j})$, which we will apply to $Z_t^K$ with $j=4$.

\begin{theorem} \label{T:TightCond}
Suppose $(\mu^K)_{K\ge1}$ is a sequence of $W^{-j}$-valued c\`adl\`ag processes.
Assume that the dynamics of $\mu^K$ are given by
\begin{equation}\label{Thm10}
(f,\mu^K_t) = (f,\mu^K_0) + \int_0^t \Lambda^K_s f ds + \tilde M^{f,K }_t,
\end{equation}
where $\tilde M^{f,K}$ is a martingale with predictable quadratic variation of the form
\begin{equation}\label{QVThm10}
\big<\tilde M^{f,K}\big>_t = \int_0^t \Gamma^K_s f ds,
\end{equation}
and $\Lambda^K_t$ and $\Gamma^K_t$ are functionals on $W^j$.
The sequence $(\mu^K)_{K\ge1}$ is tight in $\mathbb{D}(\mathbb{T},W^{-j})$
if the following conditions are satisfied:
\begin{enumerate}
\item[(T1)]
There exists $i<j$ such that
for all $t\in\mathbb{T}$,
$$\sup_{K\ge1} \mathbb{E}\big[ ||\mu_t^K||_{W^{-i}} \big] < \infty.$$
\item[(T2)]
There exists $K_0 \ge1$ such that
\begin{gather}
\sup_{K\ge K_0} \E \bigg[ \sup_{t\le T} ||\Lambda^K_t||_{W^{-j}} \bigg] \le c_T
\tag{i} \label{T2i} \\
\sup_{K\ge K_0} \E \bigg[ \sup_{t\le T} \Big|\sum_{l\ge1} \Gamma^K_t p^j_l\Big| \bigg] \le c_T.
\tag{ii} \label{T2ii}
\end{gather}
\end{enumerate}
\end{theorem}
This can be proved by showing that the
Aldous-Rebolledo criteria for tightness, stated below, holds.
For more details see for example \cite{Ald78} and \cite[pp. 34-35]{JofMet86}.

\begin{lemma}[Aldous-Rebolledo] \label{L:AldousReb}
Let $H$ be a separable Hilbert space.
A sequence $(\mu^K)_{K\ge1}$ of $H$-valued c\`adl\`ag processes is
tight in $\mathbb{D}(\mathbb{T},H)$ if the following conditions are satisfied:
\begin{enumerate}
\item[(A)]
For every $t\in\T$, $(\mu^K_t)_{K\ge1}$ is tight in $H$.
\item[(B)]
For each $\epsilon_1, \epsilon_2 >0$, there exist $\delta>0$ and $K_0 \ge1$ such that
for every sequence of stopping times $\tau^K \le T$,
$$\sup_{K>K_0} \sup_{\zeta<\delta} \mathbb{P} \big(|| \mu_{(\tau^K+\zeta)\wedge T}^K - \mu_{\tau^K}^K ||_H >\epsilon_1\big) <\epsilon_2.$$
\end{enumerate}
If $\mu^K_t$ admits a semimartingale decomposition, then for (B), it is sufficient to have it for the finite variation part and the predictable quadratic variation of the martingale part.
\end{lemma}

\begin{proof}[Proof of Theorem \ref{T:TightCond}]
Note that, for $i<j$, $W^{-i} \underset{H.S.}{\hookrightarrow} W^{-j}$, thus, the closed ball
$B_{W^{-i}}(R) := \{ \mu\in W^{-i}: ||\mu||_{W^{-i}} \le R\}$
is compact in $W^{-j}$.
Also,
$$\P(\mu^K_t \notin B_{W^{-i}}(R)) = \P(||\mu^K_t||_{W^{-i}}>R) \le \frac1R \E[||\mu^K_t||_{W^{-i}}].$$
Therefore, if (T1) holds, there exists a compact set $\mathcal{C}_\epsilon$ such that
$\P(\mu^K_t \notin \mathcal{C}_\epsilon) < \epsilon$ for all $K$,
which in turn implies (A).

Next, we show that (T2) implies (B).
Since $\mu^K$ has the form $\mu^K_t = V^K_t + \tilde M^K_t$, it remains to show (B) for $V^K_t$ and   predictable quadratic variation $\big<\big<\tilde M^K\big>\big>_t$, where $\big<\big<\tilde{M}^K\big>\big>_t$ is defined such that $\big(||\tilde M^K_t||_{W^{-j}}^2 - \big<\big<\tilde{M}^K\big>\big>_t\big)_{t\in\mathbb{T}}$ is a martingale.

To obatin (B) for $V^K_t$,  observe that by \eqref{Thm10}
\begin{multline*}
\big|(f, V_{(\tau^K+\zeta)\wedge T}^K - V_{\tau^K}^K)\big|
= \bigg| \int_{\tau^K}^{(\tau^K+\zeta)\wedge T} \Lambda^K_s f ds \bigg| \\
\le \int_{\tau^K}^{(\tau^K+\zeta)\wedge T} |\Lambda^K_s f| ds
\le \int_{\tau^K}^{(\tau^K+\zeta)\wedge T} \big|\big|\Lambda^K_s\big|\big|_{W^{-j}} ||f||_{W^{j}} ds.
\end{multline*}
Hence
\begin{multline*}
\big|\big| V_{(\tau^K+\zeta)\wedge T}^K - V_{\tau^K}^K \big|\big|_{W^{-j}}
\le \int_{\tau^K}^{(\tau^K+\zeta)\wedge T} \big|\big|\Lambda^K_s\big|\big|_{W^{-j}} ds \\
= \int_0^\zeta \big|\big|\Lambda^K_{(\tau^K+s)\wedge T}\big|\big|_{W^{-j}} ds
\le \delta \sup_{t\le T} \big|\big|\Lambda^K_t \big|\big|_{W^{-j}}.
\end{multline*}
(B) now follows from condition (T2)(i) by
  Markov's inequality.

Write $p_l$ for $p^j_l$.
 Since $||\tilde M^K_t||^2_{W^{-j}} = \sum_{l\ge1} (\tilde M_t^{p_l,K})^2$
by the Riesz Representation Theorem and Parseval's Identity,
we have $\big<\big<\tilde{M}^K\big>\big>_t = \sum_{l\ge1} \big<\tilde M^{p_l,K}\big>_t$.
To obtain (B) for $\big<\big<\tilde{M}^K\big>\big>_t$, by \eqref{QVThm10}, we have
\begin{align*}
&\big| \big<\big<\tilde{M}^K\big>\big>_{(\tau^K+\zeta)\wedge T} - \big<\big<\tilde{M}^K\big>\big>_{\tau^K} \big|
= \bigg|\sum_{l\ge1} \big<\tilde{M}^{p_l,K}\big>_{(\tau^K+\zeta)\wedge T} - \sum_{l\ge1} \big<\tilde{M}^{p_l,K}\big>_{\tau^K}\bigg| \\
&\q\q\q\q\q\q= \bigg|\sum_{l\ge1} \int_{\tau^K}^{(\tau^K+\zeta)\wedge T} \Gamma^K_s p_l \, ds \bigg|
= \bigg|\sum_{l\ge1} \int_0^\zeta \Gamma^K_{(\tau^K+s)\wedge T} p_l \, ds \bigg|
\end{align*}
and taking expectation,
\begin{multline*}
\E\Big[ \big| \big<\big<\tilde{M}^K\big>\big>_{(\tau^K+\zeta)\wedge T} - \big<\big<\tilde{M}^K\big>\big>_{\tau^K} \big| \Big]
\le \int_0^\zeta \E\bigg[ \Big|\sum_{l\ge1} \Gamma^K_{(\tau^K+s)\wedge T} p_l \Big| \bigg] ds \\
\le \int_0^\zeta \E\bigg[ \sup_{t\le T} \Big|\sum_{l\ge1} \Gamma^K_{t} p_l \Big| \bigg] ds
\le \delta \E\bigg[ \sup_{t\le T} \Big|\sum_{l\ge1} \Gamma^K_{t} p_l \Big| \bigg].
\end{multline*}
(B) now follows from condition (T2)(ii) by
  Markov's inequality.
\end{proof}

The rest of the proof consists of
 checking conditions (T1) and (T2) in space $W^{-4}$. The proof is involved and requires somewhat different representations for $Z_t^K$, and is split into sections.

\subsection{Representation for $Z_t^K$}
As representation \eqref{E:fZK} involves the unbounded derivative operator ($f\rightarrow f'$), we extend \eqref{E:fZK} to functions of two variables $f(x,s) \equiv f_s(x)$
and apply the extension to the special case $f(x,s)=\phi(x+t-s)$ (for some fixed $t$ and some function $\phi$).
This results in the removal of the derivative operator.

From \eqref{E:AbarK} and \eqref{E:LLN}, we have, for test function of two variables  $f\in C^{1,1}(\mathbb{T}^*\times\mathbb{T})$ and $t\in\mathbb{T}$,
\begin{multline} \label{E:fZK2}
(f_t,Z_t^K) = (f_0,Z_0^K) 
+ \sqrt{K} \int_0^t \left( - (h_{\bar A_s^K}^K - h^\infty_{\bar{A}_s}) f_s + f_s(0) (n_{\bar A_s^K}^K - n^\infty_{\bar{A}_s}) , \bar{A}_s \right) ds \\
+ \int_0^t \left( \partial_1f_s + \partial_2f_s - h_{\bar A_s^K}^K f_s + f_s(0) n_{\bar A_s^K}^K , Z_s^K \right) ds + \tilde{M}_t^{f,K},
\end{multline}
where $\tilde{M}_t^{f,K}$ is a
martingale with predictable quadratic variation
\begin{equation}\label{QVMGZt}
\big<\tilde{M}^{f,K}\big>_t = \int_0^t \big( f^2_s(0) w^K_{\bar A^K_s} + h^K_{\bar A^K_s}f^2_s - 2f_s(0)h^K_{\bar A^K_s}\widehat m^K_{\bar A^K_s}f_s, \bar A^K_s \big) ds.
\end{equation}
As explained above, applying \eqref{E:fZK2} to
$$f(x,s) = \phi(x+t-s) = :\varTheta_{t-s}\phi(x)$$
(for a fixed $t$) makes the term $\partial_1f(x,s) + \partial_2f(x,s)$ vanish.

Next, we obtain a representation for the corresponding martingale $\tilde{M}_t^{f,K}$.
Define the  measure $M_t$ as
\begin{align}\label{Mt}
M_t(dx) &= \delta_0(dx) \bigg(\widecheck B([0,t]) - \int_0^t (b_{A_s}\widecheck m_{A_s}, A_s) ds\bigg) \\
&\qquad+ \delta_0(dx) \bigg(\widehat B([0,t]) - \int_0^t (h_{A_s}\widehat m_{A_s}, A_s) ds\bigg) \nonumber\\
&\qquad- \bigg(\sum_{y\in I} \delta_{\lambda_y}(dx) \mathbf{1}_{\sigma_y\le t} - \int_0^t A^h_s(dx)ds \bigg),\nonumber
\end{align}
where
$$A^h_t(dx) = \sum_{y\in I} \delta_{t-\tau_y}(dx) h_{A_t}(x) \mathbf{1}_{\tau_y\le t<\sigma_y}.$$
By direct calculations, it can be seen that the martingale  $M^f_t$ in
\eqref{E:Mtf} is precisely the integral of $f$ with respect to $M_t$, i.e. $M^f_t = (f, M_t)$.
It is easy to extend  the definition of the integral to functions of two variables $f\in C(\mathbb{T}^*\times\mathbb{T})$  so that $\int_0^t \big(f_s,dM_s\big)$ coincides with $M^f_t$   in   \eqref{E:BasicEqnM2}.
Indeed, since $((g,M_t))_{t\in\mathbb{T}}$ is a martingale for any $g \in C(\mathbb{T}^*)$,
for any   $\varphi \in C(\mathbb{T})$,
the integral $\int_0^t \varphi(s) d(g,M_s)$, $t\in\mathbb{T}$,
is a well-defined martingale with predictable quadratic variation
$$\bigg<\int_0^\cdot \varphi(s) d(g,M_s)\bigg>_t
= \int_0^t \varphi^2(s) \big(g^2(0)w_{A_s} + h_{A_s}g^2 - 2g(0)h_{A_s}\widehat m_{A_s} g, A_s\big) ds.$$
Write $\int_0^t \big(\varphi(s)g,dM_s\big)$ for $\int_0^t \varphi(s) d(g,M_s)$. The extension to an arbitrary $f\in C(\mathbb{T}^*\times\mathbb{T})$ is obtained by the usual
application of the Monotone Class Theorem (e.g. \cite[I.22.1]{DelMey78}).

Let $\tilde M^K = \frac{1}{\sqrt{K}}M^K$. Since, for a fixed $t\in\mathbb{T}$, the function $f(x,s)=\varTheta_{t-s}\phi(x)$ satisfies $f(x,t)=\phi(x)$, \eqref{E:fZK2} reduces to \eqref{E:phiZK} below.

\begin{corollary}
For $\phi \in C^1$ and $t\in\mathbb{T}$,
\begin{multline} \label{E:phiZK}
(\phi,Z^K_t) = (\varTheta_t\phi,Z^K_0) \\
+ \sqrt{K} \int_0^t \big( -(h^K_{\bar A^K_s}-h^\infty_{\bar A_s})\varTheta_{t-s}\phi + \varTheta_{t-s}\phi(0)(n^K_{\bar A^K_s}-n^\infty_{\bar A_s}), \bar A_s \big) ds \\
+ \int_0^t \big( - h^K_{\bar A^K_s}\varTheta_{t-s}\phi + \varTheta_{t-s}\phi(0)n^K_{\bar A^K_s}, Z^K_s \big) ds
+ \int_0^t \big(\varTheta_{t-s}\phi,d\tilde M^K_s\big).
\end{multline}
\end{corollary}

The main step in proving tightness is the following bound.

\subsection{Boundedness of $\E \big[ ||Z^K_t||_{W^{-2}} \big]$}

\begin{proposition} \label{P:ZW-2Bound}
 $$\sup_{t\le T} \sup_{K\ge1} \E \big[ ||Z^K_t||_{W^{-2}} \big] <\infty.$$
\end{proposition}

We remark that Proposition \ref{P:ZW-2Bound} remains true with the norm taken in $C^{-1}$.
However, for the ease of presentation (as we work with spaces $W^\cdot$ mostly throughout the paper), we prove the result for $W^{-2}$, which is sufficient for our purpose.
The proof is done using representation  \eqref{E:phiZK} with $\phi\in W^2$. Each term on the RHS is dealt with separately using successive bounds.

First, we need to overcome the fact that the functions $\phi$ and $\theta_t\phi$ are defined on different domains, $\mathbb{T}^*=[0,T^*]$ and $[0,T^*-t]$, respectively.
The following lemma constructs an extension of $\theta_t\phi$ to $\mathbb{T}^*$ in a way that controls the norm.

\begin{lemma} \label{L:ShiftExt}
Let  $\phi \in W^j$ for some $j\in\N$ and
  $t\in \mathbb{T}$ be fixed.
There exists a function $\psi: \mathbb{T}^* \to \R$ such that
$\psi(x) = \phi(x+t)$ for $x\in [0,T^*-t]$,
and $\psi\in W^j$ with
$||\psi||_{W^j} \le c||\phi||_{W^j}$, where $c$ is a constant that depends on $T^*$ and $j$, but independent of $\phi$.
\end{lemma}
\begin{proof}
Take $\psi$ such that $\psi(x) = \phi(x+t)$ for $x\in [0,T^*-t]$,
and $\psi^{(j-1)}(x) = \phi^{(j-1)}(2(T^*-t)-x+t)$ for $x\in (T^*-t,T^*]$.
That is, $\psi$ is extended by reflecting the $(j-1)$th derivative along $x=T^*-t$.
Then, $\psi^{(i)}$ is continuous for $i=0,1,\dots,j-1$. Note that
$\psi^{(j)}$ does not exist at $x=T^*-t$, unless $\phi^{(j)}(T^*) =0$.

It remains to show that $||\psi||_{W^j} \le c||\phi||_{W^j}$.
For $i=j-1,j$,
\begin{align*}
\int_\mathbb{T^*} \big(\psi^{(i)}(x)\big)^2 dx
&= \int_0^{T^*-t} \big(\psi^{(i)}(x)\big)^2 dx + \int_{T^*-t}^{T^*} \big(\psi^{(i)}(x)\big)^2 dx \\
&= \int_t^{T^*} \big(\phi^{(i)}(x)\big)^2 dx + \int_{T^*-t}^{T^*} \big(\phi^{(i)}(2(T^*-t)-x+t)\big)^2 dx.
\end{align*}
For $i=0,1,\dots,j-2$,
\begin{align*}
\int_\mathbb{T^*} \big(\psi^{(i)}(x)\big)^2 dx
&= \int_0^{T^*-t} \big(\psi^{(i)}(x)\big)^2 dx + \int_{T^*-t}^{T^*} \big(\psi^{(i)}(x)\big)^2 dx \\
&= \int_t^{T^*} \big(\phi^{(i)}(x)\big)^2 dx + \int_{T^*-t}^{T^*} \big(\psi^{(i)}(x)\big)^2 dx.
\end{align*}
For the last integral, note that for $x\in (T^*-t,T^*]$,
\begin{align*}
\psi^{(i)}(x) &= \psi^{(i)}(T^*-t) + \int_{T^*-t}^x \psi^{(i+1)}(y) dy \\
&= \phi^{(i)}(T^*) + \int_{T^*-t}^x \psi^{(i+1)}(y) dy,
\end{align*}
which can be obtained recursively and be expressed in terms of $\phi$.
Finally, as $\phi \in W^j$, we have $\phi\in C^{j-1}$
and $||\phi||_{C^{j-1}} = \max_{0\le i\le j-1} \sup_{x\in\mathbb{T}^*} |\phi^{(i)}(x)| < \infty.$
Thus, with $||\phi||_{C^{j-1}} \le ||\phi||_{W^j}$ and that $\mathbb{T}^*$ is a bounded interval,
we can bound $||\psi||_{W^j}$ in terms of $T^*$ and $||\phi||_{W^j}$
and write $||\psi||_{W^j} \le c||\phi||_{W^j}$.
\end{proof}

In the sequel, $\varTheta_t\phi$ will refer to its own extension to $\mathbb{T}^*$.
We immediately get the following inequalities:
 \begin{equation} \label{extend}
 ||\varTheta_t\phi||_{W^j} \le c ||\phi||_{W^j},\q
\mbox{and for any }  x\in\mathbb{T}^*, \; |\varTheta_t\phi (x)| \le c ||\phi||_{W^j}.
\end{equation}

Next, we give some bounds that are useful in proving Proposition \ref{P:ZW-2Bound}.

 \begin{proposition} \label{P:ParaBound}
Suppose (A2) and (A3) hold. Then,
for $t\in \mathbb{T}$ and for all $x\in \mathbb{T}^*$,
$$\sqrt{K} \big|h^K_{\bar A^K_t} - h^\infty_{\bar A_t}\big|(x)
\le c (1+ ||Z^K_t||_{W^{-4}})$$
and
$$\sqrt{K} \big|n^K_{\bar A^K_t} - n^\infty_{\bar A_t}\big|(x)
\le c (1+ ||Z^K_t||_{W^{-4}}).$$
\end{proposition}

\begin{proof}
We prove only the first inequality, as the second is similar. By the triangle inequality,
\begin{align*}
\big|h^K_{\bar A^K_t} - h^\infty_{\bar A_t}\big|(x)
&\le ||h^K_{\bar A^K_t} - h^\infty_{\bar A^K_t}||_\infty + ||h^\infty_{\bar A^K_t} - h^\infty_{\bar A_t} - \partial_A h^\infty_{\bar A_t}(\bar A^K_t-\bar A_t)||_\infty \\
&\q\q+ ||\partial_A h^\infty_{\bar A_t}||_{\mathbb{L}^{-4}} ||\bar A^K_t-\bar A_t||_{W^{-4}}.
\end{align*}
Multiplying by $\sqrt{K}$ and with some manipulation, we have
\begin{align*}
&\sqrt{K}\big|h^K_{\bar A^K_t} - h^\infty_{\bar A_t}\big|(x)
\le \sqrt{K} \sup_A ||h^K_{A} - h^\infty_{A}||_\infty \\
&\q+ \frac{||Z^K_t||_{W^{-4}}}{||\bar A^K_t-\bar A_t||_{W^{-4}}} ||h^\infty_{\bar A^K_t} - h^\infty_{\bar A_t} - \partial_A h^\infty_{\bar A_t}(\bar A^K_t-\bar A_t)||_\infty
+ c_1 ||Z^K_t||_{W^{-4}},
\end{align*}
where the bound in the last term is due to (A3).
It then follows by (A2) and (A3) that
$\sqrt{K}\big|h^K_{\bar A^K_t} - h^\infty_{\bar A_t}\big|(x)
\le c_2+c_3||Z^K_t||_{W^{-4}}.$
\end{proof}

The following result follows immediately from Proposition \ref{P:ParaBound}.

\begin{proposition} \label{P:KDiffLBnd}
Suppose (A2) and (A3) hold.
For any $f\in W^j$, $j\in\N$, and $t\in \mathbb{T}$,
$$\sup_{x\in\mathbb{T}^*} \big|\sqrt{K}(L^K_{\bar A^K_t}-L^\infty_{\bar A_t}) f\big|(x) \le c (1+ ||Z^K_t||_{W^{-4}}) ||f||_{W^j}.$$
\end{proposition}

As the operator $L^K_A$ maps $W^j$ into $W^{j-1}$, we introduce
$\widehat L^K_Af = - h^K_Af + f(0) n^K_A$, so that $L^K_Af = f' + \widehat L^K_Af$,
and let $\mathcal{L}^{j,k}=L(W^j,W^k)$, the space of linear operators from $W^j$ to $W^k$.

\begin{proposition} \label{P:LKNormBnd}
Suppose (A2) holds. Then,
\begin{gather}
\tag{i} \label{E:LhatKBnd} \sup_{K,A} ||\widehat L^K_A||_{\mathcal{L}^{j,j}} \le c, \q j\le 3; \\
\tag{ii} \label{E:LKBnd} \sup_{K,A} ||L^K_A||_{\mathcal{L}^{j,j-1}} \le c, \q 2\le j\le4.
\end{gather}
\end{proposition}

\begin{proof}
For $f\in W^j$, using triangle inequality and  \eqref{E:fgNorm},
\begin{align*}
||\widehat L^K_A f||_{W^{j}}
&\le ||h^K_A||_{C^{j}}||f||_{W^{j}} + ||f||_{W^{j}} ||n^K_A||_{W^{j}}
\le c_1 ||f||_{W^{j}}
\end{align*}
due to embedding and (A2).
Thus, \eqref{E:LhatKBnd} follows.
For \eqref{E:LKBnd},
$$||L^K_A f||_{W^{j-1}}
= ||f'||_{W^{j-1}} + ||\widehat L^K_Af||_{W^{j-1}}
\le ||f||_{W^{j}} + c_1 ||f||_{W^{j-1}}
\le c_2 ||f||_{W^{j}},$$
by \eqref{E:LhatKBnd} and embedding.
Thus, \eqref{E:LKBnd} follows.
\end{proof}

Recall also the following bounds, obtained in    \cite{HamJagKle13}:
\begin{gather}
(1,\bar A_t) \le (1,\bar A_0)e^{ct} \label{E:1AInfty}, \\
\E[(1,\bar A^K_t)] \le (1,\bar A^K_0)e^{ct} \label{E:1AK}.
\end{gather}

\begin{proof} [Proof of Proposition \ref{P:ZW-2Bound}]
Let $\phi\in W^2$.
We bound each term on the RHS of  \eqref{E:phiZK}, and use repeatedly \eqref{extend}.  For the first term,
$$|(\varTheta_t\phi,Z_0^K)|
\le ||\varTheta_t\phi||_{W^2} ||Z_0^K||_{W^{-2}}
\le c_1 ||\phi||_{W^2} ||Z_0^K||_{W^{-2}}.$$
For the second term, with Proposition \ref{P:KDiffLBnd},
\begin{align*}
\Big| \sqrt{K} \int_0^t \big( (L^K_{\bar A^K_s}-L^\infty_{\bar A_s})\varTheta_{t-s}\phi, \bar A_s \big) ds \Big|
&\le c_2 \int_0^t (1+||Z^K_s||_{W^{-4}}) ||\varTheta_{t-s}\phi||_{W^2}(1,\bar A_s) ds \\
&\le c_3 ||\phi||_{W^2} (1,\bar A_0) e^{c_4t} \int_0^t (1+||Z^K_s||_{W^{-2}}) ds
\end{align*}
by \eqref{E:1AInfty} and the embedding $W^{-2}\hookrightarrow W^{-4}$.
For the third term, by Proposition \ref{P:LKNormBnd}\eqref{E:LhatKBnd},
\begin{multline*}
\Big| \int_0^t \big( -\widehat L^K_{\bar A^K_s}\varTheta_{t-s}\phi , Z^K_s \big) ds \Big|
\le c_5 \int_0^t ||\widehat L^K_{\bar A^K_s}\varTheta_{t-s}\phi||_{W^2} ||Z^K_s||_{W^{-2}} ds \\
\le c_5 \int_0^t ||\widehat L^K_{\bar A^K_s}||_{\mathcal{L}^{2,2}} ||\varTheta_{t-s}\phi||_{W^2} ||Z^K_s||_{W^{-2}} ds
\le c_6||\phi||_{W^2} \int_0^t ||Z^K_s||_{W^{-2}} ds.
\end{multline*}
For the forth term, we write $\int_0^t \varTheta^*_{t-s}d\tilde M^K_s$ for the map
$f\mapsto\int_0^t (\varTheta_{t-s}f, d\tilde M^K_s)$.
Then,
$$\Big| \int_0^t (\varTheta_{t-s}\phi, d\tilde M^K_s) \Big| \le ||\phi||_{W^2} \Big|\Big|\int_0^t (\varTheta_{t-s}^* d\tilde M^K_s)\Big|\Big|_{W^{-2}}.$$
Note that $\big(\int_0^t (\varTheta_{t-s}f, d\tilde M^K_s)\big)_{t\in\mathbb{T}}$ is not a martingale,
but for each fixed $t$, $\big(\int_0^r (\varTheta_{t-s}f, d\tilde M^K_s)\big)_{r\in\T}$ is.
Let $t\in\mathbb{T}$ be fixed.
For $r\le t$,
by the Riesz Representation Theorem and Parseval's Identity,
\begin{gather*}
\E\bigg[ \Big|\Big| \int_0^r \varTheta^*_{t-s}d\tilde M^K_s \Big|\Big|^2_{W^{-2}} \bigg]
= \E\bigg[ \sum_{l\ge1} \Big( \int_0^r (\varTheta_{t-s}p^2_l, d\tilde M^K_s) \Big)^2 \bigg]
= \sum_{l\ge1} \E\bigg[ \Big< \int_0^\cdot (\varTheta_{t-s}p^2_l, d\tilde M^K_s) \Big>_r \bigg] \\
= \sum_{l\ge1} \E\bigg[ \int_0^r \Big( (\varTheta_{t-s}p^2_l(0))^2 w_{\bar A_s^K}^K + h_{\bar A_s^K}^K (\varTheta_{t-s}p^2_l)^2
- 2\varTheta_{t-s}p^2_l(0) h_{\bar A_s^K}^K \widehat m_{\bar A_s^K}^K \varTheta_{t-s}p^2_l , \bar{A}_s^K\Big) ds \bigg].
\end{gather*}
It then follows from  \eqref{E:CONB}, (C0) and \eqref{E:1AK} that
this quantity is bounded by $c_7 (1, \bar{A}_0^K) e^{c_8r}r$.
Taking $r=t$, we have
\begin{equation} \label{E:ShiftMNorm}
\E\bigg[ \Big|\Big| \int_0^t \varTheta^*_{t-s}d\tilde M^K_s \Big|\Big|^2_{W^{-2}} \bigg]
\le c_7 (1, \bar{A}_0^K) e^{c_8t}t.
\end{equation}

Now, putting all together with triangle inequality,
\begin{multline*}
|(\phi,Z^K_t)| \le c_9\Big\{||\phi||_{W^{2}}||Z^K_0||_{W^{-2}} + ||\phi||_{W^{2}} (1,\bar A_0) e^{c_4t} \int_0^t (1+||Z^K_s||_{W^{-2}}) ds \\
+ ||\phi||_{W^{2}} \int_0^t ||Z^K_s||_{W^{-2}} ds\Big\} + ||\phi||_{W^{2}} \Big|\Big| \int_0^t \varTheta^*_{t-s}d\tilde M^K_s \Big|\Big|_{W^{-2}}.
\end{multline*}
This gives a bound to $||Z^K_t||_{W^{-2}}$.
Taking expectation and using \eqref{E:ShiftMNorm}, we have, for $t\le T$,
\begin{multline*}
\E\big[||Z^K_t||_{W^{-2}}\big] \le c_{10}\Big\{||Z^K_0||_{W^{-2}} + (1,\bar A_0) e^{c_4t} t \\
+ \big(1+ (1,\bar A_0) e^{c_4t}\big) \int_0^t \E\big[||Z^K_s||_{W^{-2}}\big] ds + (1, \bar{A}_0^K)^{1/2} e^{c_{11}t}t^{1/2}\Big\}.
\end{multline*}
It follows by Gronwall's inequality that
\begin{align*}
\E\big[||Z^K_t||_{W^{-2}}\big]
\le c_{10}\big\{||Z^K_0||_{W^{-2}} + (1,\bar A_0) e^{c_4T}T + (1, \bar{A}_0^K)^{1/2} e^{c_{11}T}T^{1/2} \big\}
e^{c_{10}(1+ (1,\bar A_0) e^{c_4T})t}.
\end{align*}
Finally, taking supremum over $t$ and $K$, this quantity is finite
due to (A4) and (C3).
\end{proof}
\subsection{Proof of tightness} \label{S:T1T2}

It remains to check the tightness condition (T2), as (T1) holds by Proposition \ref{P:ZW-2Bound}.
The conditions \eqref{T2i} and \eqref{T2ii} are verified in a few steps.
Proceeding from Theorem \ref{T:TightCond}, we let
$$\Lambda^K_t f = \sqrt{K}\big( (L^K_{\bar A^K_t}-L^\infty_{\bar A_t}) f, \bar A_t\big)
+ \big(L^K_{\bar A^K_t}f, Z^K_t\big)$$
and
$$\Gamma^K_t f = \big(f^2(0) w^K_{\bar A^K_t} + h^K_{\bar A^K_t} f^2 - 2f(0) h^K_{\bar A^K_t}\widehat m^K_{\bar A^K_t} f, \bar A^K_t\big).$$

\begin{proposition} \label{P:Lambdaf}
Let $j\in\{3,4\}$. For $f\in W^j$,
$$|\Lambda^K_t f| \le c_1 ||f||_{W^{j}} \big(1+ (1, \bar A_0)e^{c_2t} \big) \big(1+ ||Z^K_t||_{W^{-(j-1)}} \big).$$
\end{proposition}

\begin{proof}
For $f\in W^j$, we have $L^K_{A^K_t}f \in W^{j-1}$ and
\begin{align*}
|\Lambda^K_t f|
&\le \big| \sqrt{K}\big( (L^K_{\bar A^K_t}-L^\infty_{\bar A_t}) f, \bar A_t\big) \big|
+ \big| \big(L^K_{\bar A^K_t}f, Z^K_t\big) \big| \\
&\le \big(| \sqrt{K}(L^K_{\bar A^K_t}-L^\infty_{\bar A_t}) f|, \bar A_t\big)
+ \big|\big| L_{\bar A_t^K}^K \big|\big|_{\mathcal{L}^{j,j-1}} ||f||_{W^j} ||Z_t^K||_{W^{-(j-1)}} \\
&\le c_1 \big(1+||Z^K_t||_{W^{-4}}\big) ||f||_{W^{j}} (1, \bar A_t)
+ c_2 ||f||_{W^j} ||Z_t^K||_{W^{-(j-1)}}
\end{align*}
due to Propositions \ref{P:KDiffLBnd} and \ref{P:LKNormBnd}.
Then, by \eqref{E:1AInfty} and the embedding $W^{-(j-1)}\hookrightarrow W^{-j}$,
$$|\Lambda^K_t f| \le c_3 \big(1+||Z^K_t||_{W^{-(j-1)}}\big) ||f||_{W^{j}} (1, \bar A_0)e^{c_4t}
+ c_2 ||f||_{W^{j}} ||Z_t^K||_{W^{-(j-1)}}.$$
The statement now follows by simple algebra.
\end{proof}

\begin{proposition} \label{P:SumGampl}
$$\Big|\sum_{l\ge1} \Gamma^K_t p^j_l\Big| \le c_1(1, \bar A^K_t).$$
\end{proposition}

\begin{proof}
This follows directly from (C0) and  \eqref{E:CONB}.
\end{proof}

\begin{proposition} \label{P:ZW-3Bound}
$$\sup_{K\ge1} \mathbb{E} \Big[ \sup_{t\le T} ||Z^K_t||_{W^{-3}} \Big] < \infty.$$
\end{proposition}
\begin{proof}
Let $f\in W^{3}$.
Using Proposition \ref{P:Lambdaf}, we have
\begin{align*}
&|(f,Z^K_t)|
\le ||f||_{W^{3}} ||Z^K_0||_{W^{-3}} \\
&\q+ c_1 ||f||_{W^{3}} \big(1+ (1, \bar A_0)e^{c_2t} \big) \int_0^t  \big(1+ ||Z^K_s||_{W^{-2}} \big) ds
+ ||f||_{W^{3}} ||\tilde{M}_t^{K}||_{W^{-3}}.
\end{align*}
This gives a bound to $||Z^K_t||_{W^{-3}}$ and consequently,
\begin{multline} \label{E:supZKW-3}
\sup_{t\le T} ||Z^K_t||_{W^{-3}} \le ||Z^K_0||_{W^{-3}} \\
+ c_1 \big(1+ (1, \bar A_0)e^{c_2T} \big) \int_0^T  \big(1+ ||Z^K_s||_{W^{-2}} \big) ds + \sup_{t\le T} ||\tilde{M}_t^{K}||_{W^{-3}}.
\end{multline}
Now, by the Riesz Representation Theorem and Parseval's Identity, we have
\begin{multline*}
\E\bigg[ \sup_{t\le T} ||\tilde{M}_t^{K}||^2_{W^{-3}} \bigg]
= \E\bigg[ \sup_{t\le T} \sum_{l\ge1} (\tilde M^{p^3_l,K}_t)^2 \bigg] \\
\le \E\bigg[ \sum_{l\ge1} \sup_{t\le T} (\tilde M^{p^3_l,K}_t)^2 \bigg]
\le 4 \sum_{l\ge1} \E\Big[ \big<\tilde M^{p^3_l,K}\big>_T \Big]
\end{multline*}
using Doob's inequality.
It then follows by Proposition \ref{P:SumGampl} and inequality \eqref{E:1AK} that
$$\E\bigg[ \sup_{t\le T} ||\tilde{M}_t^{K}||^2_{W^{-3}} \bigg]
\le 4 \sum_{l\ge1} \E\bigg[ \int_0^T \Gamma^K_s p^3_l ds \bigg]
\le c_5(1, \bar A^K_0) e^{c_6T}T.$$
Therefore, taking expectation in \eqref{E:supZKW-3}, we obtain
\begin{multline} \label{E:WhyCantGoLower}
\E\bigg[\sup_{t\le T} ||Z^K_t||_{W^{-3}}\bigg]
\le c_7 \bigg\{ ||Z^K_0||_{W^{-3}} \\
\q+ \big(1+ (1, \bar A_0)e^{c_2T} \big) \int_0^T  \Big(1+ \E\big[||Z^K_s||_{W^{-2}}\big] \Big) ds
+ (1, \bar A^K_0)^{1/2} e^{c_8T}T^{1/2} \bigg\}.
\end{multline}
Noting that $\E[ ||Z^K_s||_{W^{-2}}]$ is bounded by Proposition \ref{P:ZW-2Bound},
and using (A4) and (C3), complete the proof.
\end{proof}

\begin{proposition} \label{P:CheckT2}
Conditions \eqref{T2i} and \eqref{T2ii} of (T2) hold for $W^{-4}$, namely
\begin{gather}
\sup_{K\ge1} \E \bigg[ \sup_{t\le T} ||\Lambda^K_t||_{W^{-4}} \bigg] \le c,
\tag{i} \\
\sup_{K\ge1} \E \bigg[ \sup_{t\le T} \Big|\sum_{l\ge1} \Gamma^K_t p^4_l\Big| \bigg] \le c.
\tag{ii}
\end{gather}
\end{proposition}

\begin{proof}
From Proposition \ref{P:Lambdaf} with $j=4$, we have
$$||\Lambda^K_t||_{W^{-4}} \le c_1 \big(1+ (1, \bar A_0)e^{c_2t} \big) \big(1+ ||Z^K_t||_{W^{-3}} \big).$$
Taking supremum over $t\le T$ and expectation, we have
$$\E \bigg[ \sup_{t\le T} ||\Lambda^K_t||_{W^{-4}} \bigg]
\le c_1 \big(1+ (1, \bar A_0)e^{c_2T} \big) \bigg(1+ \E \bigg[ \sup_{t\le T}||Z^K_t||_{W^{-3}} \bigg] \bigg),$$
which is bounded in $K$ by Proposition \ref{P:ZW-3Bound}.
Thus, condition \eqref{T2i} holds.

Now we verify condition \eqref{T2ii}.
From Proposition \ref{P:SumGampl},
$$\E \bigg[ \sup_{t\le T} \Big|\sum_{l\ge1} \Gamma^K_t p^4_l\Big| \bigg]
\le c_3\E \bigg[ \sup_{t\le T}(1, \bar A^K_t) \bigg].$$
But,
$$(1,\bar A^K_t) \le (1,\bar A^K_0) + c_4 \int_0^t (1,\bar A^K_s) ds + \frac{1}{\sqrt{K}} \tilde M^{1,K}_t$$
and for $S\le T$,
$$\E\bigg[ \sup_{t\le S} (1,\bar A^K_t) \bigg]
\le (1,\bar A^K_0) + c_4 \int_0^S \E\bigg[ \sup_{u\le s} (1,\bar A^K_u)\bigg] ds + \frac{1}{\sqrt{K}}\E\bigg[ \sup_{t\le T} \tilde M^{1,K}_t \bigg].$$
It follows by Gronwall's inequality that
$$\E\bigg[ \sup_{t\le T} (1,\bar A^K_t) \bigg]
\le \bigg\{ (1,\bar A^K_0) + \frac{1}{\sqrt{K}} \E\bigg[ \sup_{t\le T} \tilde M^{1,K}_t \bigg] \bigg\} e^{c_4T},$$
where by Doob's inequality,
\begin{align*}
\E\bigg[ \sup_{t\le T} \tilde M^{1,K}_t \bigg]^2
&\le \E\bigg[ \sup_{t\le T} \big(\tilde M^{1,K}_t\big)^2 \bigg]
\le 4 \E\big[ \big< \tilde M^{1,K} \big>_T\big]
\le c_5(1,\bar A^K_0) e^{c_6T}.
\end{align*}
Therefore, condition \eqref{T2ii} follows, using (C3).
\end{proof}

\begin{corollary}
Both sequences $Z^K$ and $\tilde M^K$ are tight in $\D(\mathbb{T},W^{-4})$.
\end{corollary}


\subsection{C-tightness of $Z^K$ and $\tilde M^K$} \label{S:CTight}

It can be further shown that $Z^K$ and $\tilde M^K$ are C-tight,
that is, the two sequences are tight and all limit points of the sequences are continuous.

\begin{proposition} \label{P:ZKCtight}
The sequence $Z^K$ is C-tight and all limit points of $Z^K$ are elements of $\mathbb{C}(\mathbb{T},W^{-4})$.
\end{proposition}
\begin{proof}
We have established that $Z^K$ is tight,
it remains to show that (see e.g. \cite[Proposition VI 3.26(iii)]{JacShi03}),
for all $u\in\mathbb{T}$ and $\epsilon>0$,
$$\lim_{K\to\infty}\mathbb{P}\Big(\sup_{t\le u} ||\Delta Z^K_t||_{W^{-4}} >\epsilon\Big) =0.$$
Observe that $Z^K$ jumps when $A^K$ jumps, which occurs when there is a birth or a death.
Thus, for $f\in W^{4}$, we have
\begin{align*}
|(f, \Delta Z^K_t)| &= |(f,Z^K_t-Z^K_{t-})| = \frac{1}{\sqrt{K}} |(f,A^K_t-A^K_{t-})| \\
&\le \frac{1}{\sqrt{K}} \max\Big\{ \sup_{x\in\mathbb{T}^*} \widecheck\xi^K_{A^K_t}(x) |f(0)| , \sup_{x\in\mathbb{T}^*} |\widehat\xi^K_{A^K_t}(x) f(0) - f(x)| \Big\} \\
&\le \frac{1}{\sqrt{K}}||f||_{W^{4}} (1+\Xi)
\end{align*}
by (A1),
giving $||\Delta Z^K_t||_{W^{-4}} \le \frac{1}{\sqrt{K}}(1+\Xi)$. Hence,
$$\mathbb{P}\Big(\sup_{t\le u} ||\Delta Z^K_t||_{W^{-4}} >\epsilon\Big)
\le \frac1\epsilon \mathbb{E}\Big[ \sup_{t\le u} ||\Delta Z^K_t||_{W^{-4}} \Big]
\le \frac1\epsilon \frac{1}{\sqrt{K}} \big(1+ \mathbb{E}[\Xi]\big),$$
which converges to zero as $K$ tends to infinity.
\end{proof}

\begin{corollary} \label{C:MKCtight}
The sequence of martingales $\tilde M^K$ is C-tight and
all limit points of $\tilde M^K$ are elements of $\mathbb{C}(\mathbb{T},W^{-4})$.
\end{corollary}

\begin{proof}
As $Z^K$ and $\tilde M^K$ have the same discontinuities,
$\Delta Z^K_t = \Delta \tilde M^K_t$ and it follows
that $\tilde M^K$ satisfies the conditions of being C-tight.
\end{proof}

\subsection{Convergence of $\tilde M^K$ and $Z^K$} \label{SS:LimitNUniq}

\begin{proposition} \label{P:MInfty}
The sequence $\tilde M^K$ convergences weakly to $\tilde M^\infty$ such that for any $f\in W^{4}$,
$\tilde M^{f,\infty}_t \equiv (f,\tilde M^\infty_t)$, $t\in\mathbb{T}$, is a continuous Gaussian martingale
with predictable quadratic variation
\begin{equation} \label{E:MinftyPQV}
\big<\tilde M^{f,\infty}\big>_t = \int_0^t \big( f^2(0) w^\infty_{\bar A_s} + h^\infty_{\bar A_s} f^2 - 2f(0) h^\infty_{\bar A_s}\widehat m^\infty_{\bar A_s} f, \bar A_s \big) ds.
\end{equation}
\end{proposition}

\begin{proof}
Let $f\in W^{4}$.
Recall from the proof of Proposition \ref{P:ZKCtight} that
$$\sup_{s\le t} |\Delta \tilde M^{f,K}_s|^2 = \sup_{s\le t} |\Delta (f,Z^K_s)|^2
\le \frac{1}{K} ||f||^2_{W^{4}} (1+\Xi)^2.$$
Thus,
$$\sup_{K\ge1} \E\Big[ \sup_{s\le t} |\Delta \tilde M^{f,K}_s|^2 \Big]
\le \sup_{K\ge1} \E\Big[ \frac{1}{K} ||f||^2_{W^{4}} (1+\Xi)^2 \Big],$$
which is finite by (A1).
Therefore, $\sup_{s\le t} |\Delta \tilde M^{f,K}_s|$ is uniformly integrable and
converges to zero in probability for all $t\in\mathbb{T}$.
All limit points of $\tilde M^{f,K}$ are continuous (from Corollary \ref{C:MKCtight})
and $\big<\tilde M^{f,K}\big>_t$ converges to \eqref{E:MinftyPQV}.
By \cite[Theorem VIII 3.12(iv)]{JacShi03}
  $\tilde M^{f,K}$ converges to a continuous martingale $\tilde M^{f,\infty}$
with predictable quadratic variation  in \eqref{E:MinftyPQV}.
The limiting process is Gaussian as the predictable quadratic variation is deterministic.

Tightness of $\tilde M^K$ implies that there exists a subsequence that converges.
Suppose $M$ and $N$ both are accumulation points of $\tilde M^K$.
Then, we have $(f,M) = \tilde M^{f,\infty} = (f,N)$ for every $f\in W^{4}$,
and thus, we must have $M=N$ in $W^{-4}$.
Therefore, we can conclude that $\tilde M^K$ converges to $\tilde M$,
where $\tilde M^\infty$ is defined such that $(f, \tilde M^\infty) = \tilde M^{f,\infty}$
for every $f\in W^{4}$.
\end{proof}

\begin{proposition} \label{P:phiZInfty}
Every limit point $\mathscr{Z}$ of the sequence $Z^K$ satisfies,
for $\phi \in W^{4}$ and $t\in\mathbb{T}$,
\begin{multline} \label{E:phiZInfty}
(\phi,\mathscr{Z}_t) = (\varTheta_t\phi,\mathscr{Z}_0)
+ \int_0^t \big(-\partial_A h^\infty_{\bar{A}_s}(\mathscr{Z}_s) \varTheta_{t-s}\phi + \varTheta_{t-s}\phi(0)\partial_A n^\infty_{\bar{A}_s}(\mathscr{Z}_s), \bar{A}_s\big) ds \\
+ \int_0^t \big(- h^\infty_{\bar{A}_s}\varTheta_{t-s}\phi + \varTheta_{t-s}\phi(0)n^\infty_{\bar{A}_s}, \mathscr{Z}_s\big) ds
+ \int_0^t (\varTheta_{t-s}\phi, d\tilde{M}^{\infty}_s).
\end{multline}
\end{proposition}

\begin{proof}
First, we show that $\sqrt{K}(h^K_{\bar A^K_s}-h^\infty_{\bar A_s})$
converges to $\partial_A h^\infty_{\bar{A}_s}(\mathscr{Z}_s)$:
\begin{align*}
&\big| \sqrt{K}(h^K_{\bar A^K_s}-h^\infty_{\bar A_s}) - \partial_A h^\infty_{\bar{A}_s}(\mathscr{Z}_s) \big| \\
&\q\q\le \sqrt{K} \big|h^K_{\bar A^K_s} - h^\infty_{\bar A^K_s}\big|
+ \sqrt{K} \big|h^\infty_{\bar A^K_s} - h^\infty_{\bar A_s} - \partial_A h^\infty_{\bar{A}_s}(\bar A^K_s-\bar A_s)\big| \\
&\q\q\q\q\q\q+ \big|\partial_A h^\infty_{\bar{A}_s}(Z^K_s) - \partial_A h^\infty_{\bar{A}_s}(\mathscr{Z}_s)\big| \\
&\q\q\le \sqrt{K} \sup_A ||h^K_A - h^\infty_{A}||_\infty
+ \frac{||Z^K_s||_{W^{-4}}}{||\bar A^K_s-\bar A_s||_{W^{-4}}} ||h^\infty_{\bar A^K_s} - h^\infty_{\bar A_s} - \partial_A h^\infty_{\bar{A}_s}(\bar A^K_s-\bar A_s)||_\infty \\
&\q\q\q\q\q\q+ ||\partial_A h^\infty_{\bar{A}_s}||_{\mathbb{L}^{-4}} ||Z^K_s-\mathscr{Z}_s||_{W^{-4}},
\end{align*}
which converges to zero as $K$ tends to infinity;
the first term by (A2), the second by the definition of Fr\'echet derivative (A3), and the last term due to $\mathscr{Z}$ being a limit.
Similarly, $\sqrt{K}(n^K_{\bar A^K_s}-n^\infty_{\bar A_s})$
converges to $\partial_A n^\infty_{\bar{A}_s}(\mathscr{Z}_s)$.
Thus,
\begin{multline*}
\sqrt{K} \int_0^t \big( -(h^K_{\bar A^K_s}-h^\infty_{\bar A_s})\varTheta_{t-s}\phi + \varTheta_{t-s}\phi(0)(n^K_{\bar A^K_s}-n^\infty_{\bar A_s}), \bar A_s \big) ds \\
\underset{K\to\infty}{\to} \int_0^t \big(-\partial_A h^\infty_{\bar{A}_s}(\mathscr{Z}_s) \varTheta_{t-s}\phi + \varTheta_{t-s}\phi(0)\partial_A n^\infty_{\bar{A}_s}(\mathscr{Z}_s), \bar{A}_s\big) ds
\end{multline*}
by dominated convergence theorem.

Next, we show that
$\int_0^t ( - h^K_{\bar A^K_s}\varTheta_{t-s}\phi + \varTheta_{t-s}\phi(0)n^K_{\bar A^K_s}, Z^K_s ) ds$
converges to
$\int_0^t (- h^\infty_{\bar{A}_s}\varTheta_{t-s}\phi + \varTheta_{t-s}\phi(0)n^\infty_{\bar{A}_s}, \mathscr{Z}_s) ds$.
Using a similar argument as for Proposition \ref{P:LKNormBnd}\eqref{E:LhatKBnd},
with (A2),
\begin{align*}
&\Big| \int_0^t \big( - h^K_{\bar A^K_s}\varTheta_{t-s}\phi + \varTheta_{t-s}\phi(0)n^K_{\bar A^K_s}, Z^K_s \big) ds
- \int_0^t \big(- h^\infty_{\bar{A}_s}\varTheta_{t-s}\phi + \varTheta_{t-s}\phi(0)n^\infty_{\bar{A}_s}, \mathscr{Z}_s\big) ds \Big| \\
&\q\q\le \int_0^t || -(h^K_{\bar A^K_s}-h^\infty_{\bar A_s}) \varTheta_{t-s}\phi + \varTheta_{t-s}\phi(0) (n^K_{\bar A^K_s}-n^\infty_{\bar A_s}) ||_{W^{4}} ||Z^K_s||_{W^{-4}} ds \\
&\q\q\q\q\q\q+ \int_0^t ||- h^\infty_{\bar{A}_s}\varTheta_{t-s}\phi + \varTheta_{t-s}\phi(0)n^\infty_{\bar{A}_s} ||_{W^{4}} ||Z^K_s-\mathscr{Z}_s||_{W^{-4}} ds \\
&\q\q\le \int_0^t c_1 ||\phi||_{W^{4}} \big( ||h^K_{\bar A^K_s}-h^\infty_{\bar A_s}||_{C^{4}} + ||n^K_{\bar A^K_s}-n^\infty_{\bar A_s}||_{W^{4}} \big) ||Z^K_s||_{W^{-4}} ds \\
&\q\q\q\q\q\q+ \int_0^t c_2 ||\phi||_{W^{4}} \big( ||h^\infty_{\bar{A}_s}||_{C^4} + ||n^\infty_{\bar{A}_s}||_{W^4} \big) ||Z^K_s-\mathscr{Z}_s||_{W^{-4}} ds,
\end{align*}
which converges to 0 as $K\to\infty$.

Together with the convergence of $Z^K_0$ in (A4)
and the convergence of $\tilde M^K$ established in Proposition \ref{P:MInfty},
the proof is complete.
\end{proof}

It remains to show the uniqueness of the solution to Equation \eqref{E:phiZInfty}.

\begin{proposition} \label{P:ZInftyUniq}
Suppose that $\mathscr{Z}$ and $\mathscr{Y}$ both are solutions to Equation \eqref{E:phiZInfty} in Proposition \ref{P:phiZInfty} with $\mathscr{Z}_0=\mathscr{Y}_0$, then $\mathscr{Z}=\mathscr{Y}$.
\end{proposition}

\begin{proof}
First, note that Proposition \ref{P:LKNormBnd}\eqref{E:LhatKBnd} remains true if $\widehat L^K_A$ is replaced with
$\widehat L^\infty_A: f \mapsto -h^\infty_Af + f(0) n^\infty_A$, for $j\le 4$, due to (A2).
Now, let $\phi\in W^{4}$ and $t\in\mathbb{T}$, by triangle inequality, we have
\begin{align*}
|(\phi,\mathscr{Z}_t-\mathscr{Y}_t)|
&\le \int_0^t \big( \big|\partial_A h^\infty_{\bar{A}_s}(\mathscr{Z}_s-\mathscr{Y}_s)\big| |\varTheta_{t-s}\phi| + |\varTheta_{t-s}\phi(0)| \big|\partial_A n^\infty_{\bar{A}_s}(\mathscr{Z}_s-\mathscr{Y}_s)\big|, \bar{A}_s \big) ds \\
&\q\q\q+ \int_0^t ||- h^\infty_{\bar{A}_s}\varTheta_{t-s}\phi + \varTheta_{t-s}\phi(0)n^\infty_{\bar{A}_s}||_{W^{4}} ||\mathscr{Z}_s-\mathscr{Y}_s||_{W^{-4}} ds \\
&\le \int_0^t c_1 ||\phi||_{W^{4}} \big( ||\partial_A h^\infty_{\bar{A}_s}||_{\mathbb{L}^{-4}} + ||\partial_A n^\infty_{\bar{A}_s}||_{\mathbb{L}^{-4}} \big) ||\mathscr{Z}_s-\mathscr{Y}_s||_{W^{-4}} (1,\bar{A}_s) ds \\
&\q\q\q+ \int_0^t ||\widehat L^\infty_{\bar{A}_s}||_{\mathcal{L}^{4,4}} ||\varTheta_{t-s}\phi||_{W^{4}} ||\mathscr{Z}_s-\mathscr{Y}_s||_{W^{-4}} ds \\
&\le c_2 ||\phi||_{W^{4}} \big(1+ (1,\bar{A}_0) e^{c_3T}\big) \int_0^t ||\mathscr{Z}_s-\mathscr{Y}_s||_{W^{-4}} ds.
\end{align*}
Thus, 
$$||\mathscr{Z}_t-\mathscr{Y}_t||_{W^{-4}} \le c_2 \big(1+ (1,\bar{A}_0) e^{c_3T}\big) \int_0^t ||\mathscr{Z}_s-\mathscr{Y}_s||_{W^{-4}} ds.$$
It then follows by Gronwall's inequality that $||\mathscr{Z}_t-\mathscr{Y}_t||_{W^{-4}} =0$. Therefore, $\mathscr{Z}=\mathscr{Y}$.
\end{proof}

Lastly, we note that Equation \eqref{E:phiZInfty} is the same as Equation \eqref{E:fZInfty}.
This is straightforward and the proof is omitted.

\begin{proposition}\label{P:same}
The limiting process $Z$ satisfies Equation \eqref{E:fZInfty}, for any $f\in W^{4}$ and $t\in\mathbb{T}$.
\end{proposition}

\section{Proofs of Corollary \ref{C:SPDE} and Proposition \ref{P:EZtMeasure}} \label{S:ZInftyProofs}

\begin{proof} [Proof of Corollary \ref{C:SPDE}]
The SPDE representation follows by direct calculation.
To establish that $(\tilde{M}^\infty)_{t,f}$ is Gaussian, we use the Cram\'er-Wold device,
by showing that for all $f_1, \ldots, f_n$ in $W^4$,
$((\tilde M_t^{f_1,\infty}, \ldots, \tilde M_t^{f_n,\infty}))_{t\ge0}$ is Gaussian.
This is equivalent to showing that for all $\alpha_1,\ldots, \alpha_n \ge0$,
$(\alpha_1\tilde M_t^{f_1,\infty} + \cdots + \alpha_n\tilde M_t^{f_n,\infty})_{t\in\mathbb{T}}$ is Gaussian,
which is true observing that
$\alpha_1\tilde M_t^{f_1,\infty} +\cdots+ \alpha_n\tilde M_t^{f_n,\infty} = \tilde M_t^{(\alpha_1f_1+\cdots+\alpha_nf_n),\infty}$.
\end{proof}

\begin{proof} [Proof of Proposition \ref{P:EZtMeasure}]
From representation \eqref{E:phiZInfty}, we obtain, for $\phi\in W^4$,
\begin{multline*}
\E[(\phi,Z_t)] = (\varTheta_t\phi,Z_0) \\
+ \int_0^t \int \bigg(-\E\bigg[ \int g^h(\bar{A}_s,x,y) Z_s(dy) \bigg] \varTheta_{t-s}\phi(x) + \phi(t-s)\E\bigg[ \int g^n(\bar{A}_s,x,y) Z_s(dy) \bigg] \bigg) \bar{A}_s(dx) ds \\
+ \int_0^t \E\Big[\big(- h^\infty_{\bar{A}_s}\varTheta_{t-s}\phi + \phi(t-s)n^\infty_{\bar{A}_s}, Z_s\big)\Big] ds
\end{multline*}
as $\E\big[\int_0^t (\varTheta_{t-s}\phi, d\tilde{M}^{\infty}_s) \big] =0$.
Defining $\nu_t : f \mapsto \E[(f,Z_t)]$, the above becomes
\begin{multline} \label{E:PhiEZt}
(\phi,\nu_t) = (\varTheta_t\phi,\nu_0) \\
+ \int_0^t \int \left(- \int g^h(\bar{A}_s,x,y) \nu_s(dy) \varTheta_{t-s}\phi(x) + \phi(t-s) \int g^n(\bar{A}_s,x,y) \nu_s(dy) \right) \bar{A}_s(dx) ds \\
+ \int_0^t \Big(- h^\infty_{\bar{A}_s}\varTheta_{t-s}\phi + \phi(t-s)n^\infty_{\bar{A}_s}, \nu_s\Big) ds.
\end{multline}
Using \eqref{extend}, \eqref{E:1AInfty} and (A2), we have
\begin{multline*}
|(\phi,\nu_t)| \le ||\phi||_{W^4} ||\nu_0||_{W^{-4}} \\
+ c_1 ||\phi||_{W^4} \Big(\sup_{A,x} ||g^h(A,x,\cdot)||_{W^4} + \sup_{A,x} ||g^n(A,x,\cdot)||_{W^4}\Big) (1,\bar A_0)e^{c_2t} \int_0^t ||\nu_s||_{W^{-4}} ds \\
+ c_3 ||\phi||_{W^4} \int_0^t ||\nu_s||_{W^{-4}} ds,
\end{multline*}
which gives, by Gronwall's inequality,
$||\nu_t||_{W^{-4}} \le c_T ||\nu_0||_{W^{-4}}$.

Now, let $(\phi_k)_k$ be a sequence of functions in $C^\infty$ that converges to $\phi\in C^0$.
By dominated convergence theorem,
\eqref{E:PhiEZt} holds for $\phi\in C^0$.
Moreover, $\nu:C^0\to\R$ is a bounded linear operator.
Therefore, $\nu_t$ can be seen as an element in $C^{-0}$, that is, it is a signed measure.
\end{proof}


\section{Example: parameters that are essentially linear} \label{S:Example}

In this section, we give some examples of the reproduction parameters that satisfy the assumptions that we imposed for the LLN and CLT.
Suppose the reproduction parameters are of the form $q^K_{\bar A^K}(x) = q\left(x,(1,\bar A^K),\int g(x,y)\bar A^K(dy)\right)$, where $q$ could be any of $b,h,m,v$; and, 
$g:\T^*\times\T^* \to\R$ and $q:\T^*\times\R_+\times\R\to\R$.
We shall refer to the function $g$ as a demography kernel.
Suppose that:
\begin{enumerate}
\item The function $g$ is element of $C^{4,4}$.
\item The functions $b,h,m$ are elements of $C^{4,1,4}$; and for $q=b,h,m$,
\begin{enumerate}
\item $\begin{displaystyle} \sup_{x,y,z} |\partial_2 q(x,y,z)| <\infty\end{displaystyle}$;
\item $\begin{displaystyle} \sup_{x,y,z} (1+y)^k |\partial_3^k \partial_1^j q(x,y,z)| <\infty \textnormal{ for } j=0,1,\dots,4-k \textnormal{ and } k=0,1,2,3,4 \end{displaystyle}$,
\end{enumerate}
where $\partial_i^j$ denotes the $j$th order partial derivative with respect to the $i$th variable.
\item The function $v$ is bounded and Lipschitz in the second and the third variables, uniformly in the first variable, i.e.
$$\sup_x |v(x,y_1,z_1)-v(x,y_2,z_2)| \le c(|y_1-y_2|+|z_1-z_2|).$$
\end{enumerate}
Then, together with assumptions (C3), (A1) and (A4), the LLN and CLT hold with
$q^\infty_{\bar A}(x) = q\left(x,(1,\bar A),\int g(x,y) \bar A(dy)\right)$ and
\begin{multline*}
\partial_Aq^\infty_{A_0}(B)(x) 
= \partial_2q\left(x,(1,A_0),\int g(x,y) A_0(dy)\right)(1,B) \\
+ \partial_3q\left(x,(1,A_0),\int g(x,y),A_0(dy)\right) \int g(x,y) B(dy).
\end{multline*}
It also follows from Proposition \ref{P:EZtMeasure} that $\nu_t :f \mapsto \E[(f,Z_t)]$ is a measure and satisfies the following equation, with $\mathtt{x}_t:= (1,\bar A_t)$: 
\begin{align*}
&(f,\nu_t) = (f,\nu_0) \\
&+ \int_0^t \bigg\{ (1,\nu_s) \int \left(f(0)\partial_2n\Big(x,\mathtt{x}_s,\int g(x,y) \bar A_s(dy) \Big)
- \partial_2h\Big(x,\mathtt{x}_s, \int g(x,y) \bar A_s(dy)\Big)f(x) \right) \bar A_s(dx) \\
&\q+ \left(\left(f(0)\partial_3n\Big(x,\mathtt{x}_s, \int g(x,y) \bar A_s(dy)\Big)
- \partial_3h\Big(x,\mathtt{x}_s, \int g(x,y) \bar A_s(dy)\Big)f(x) \right) \int g(x,y) \nu_s(dy) \right) \bar A_s(dx) \bigg\} ds \\
&+ \int_0^t \left(f'(x)-h\Big(x,\mathtt{x}_s, \int g(x,y)\bar A_s(dy)\Big)f(x) + f(0)n\Big(x,\mathtt{x}_s, \int g(x,y)\bar A_s(dy)\Big) \right) \nu_s(dx) ds.
\end{align*}

In what follows, we consider a few special cases. We will also see that when $q_{\bar A^K}^K(x)$ is a function of $(1,\bar A^K)$ only, or is a constant, an explicit expression for the density of the measure $\E[Z_t]$ can be computed.


\subsection{Special case}

Suppose that the reproduction parameters are of the form $\bar q\left(x,\frac{\int g(x,y)\bar A^K(dy)}{1+(1,\bar A^K)}\right)$, where $ \bar q:\T^*\times\R\to\R$ and $g \in C^{4,4}$.
In other words, we take $q(x,y,z) = \bar q(x,\frac{z}{1+y})$.
Conditions (2) and (3) on $q$ above then reduce to
$\bar q \in C^{4,4}$ with
\begin{enumerate}
\item[(a)] $\begin{displaystyle}\sup_{x,u} |u \partial_2 \bar q(x,u)| < \infty\end{displaystyle}$, for $\bar q=b,h,m,v$
\item[(b)] $\begin{displaystyle}\sup_{x,u} |\partial_1^j\partial_2^k \bar q(x,u)| < \infty\end{displaystyle}$, for $j=0,1,\dots,4-k$, $k=0,1,2,3,4$ and $\bar q=b,h,m$.
\end{enumerate}
Note that (a) implies the Lipschitz condition.
Moreover,
$$\partial_A \bar q^\infty_{A_0}(B)(x) =
\partial_2 \bar q\left(x,\frac{\int g(x,y) A_0(dy)}{1+(1,A_0)}\right) \frac{(1+(1,A_0))\int g(x,y)B(dy) - (1,B)\int g(x,y) A_0(dy)}{(1+(1,A_0))^2},$$
and the measure $\nu_t : f \mapsto \E[(f,Z_t)]$ satisfies the following equation with $\mathtt{x}_t:= (1,\bar A_t)$:
\begin{align*}
&(f,\nu_t) = (f,\nu_0) + \int_0^t \bigg( \frac{(1+\mathtt{x}_s) \int g(x,y) \nu_s(dy) - (1,\nu_s) \int g(x,y) \bar A_s(dy)}{(1+\mathtt{x}_s)^2} \\
&\q\q\q\q\times \int \left(f(0) \partial_2n\Big(x,\frac{\int g(x,y) \bar A_s(dy)}{1+\mathtt{x}_s}\Big) - \partial_2h\Big(x,\frac{\int g(x,y) \bar A_s(dy)}{1+\mathtt{x}_s}\Big) f(x) \right) \bar A_s(dx) ds \\
&\q\q+ \int_0^t \int \left(f'(x) - h\Big(x,\frac{\int g(x,y) \bar A_s(dy)}{1+\mathtt{x}_s}\Big) f(x) + f(0)n\Big(x,\frac{\int g(x,y) \bar A_s(dy)}{1+\mathtt{x}_s}\Big) \right) \nu_s(dx) ds.
\end{align*}


\subsection{Age-and-density-dependent case}

Suppose that the parameters are of the form $\tilde q(x,(1,\bar A^K))$, $\tilde q:\T^*\times\R_+\to\R$,
that is, $q(x,y,z) = \tilde q(x,y)$.
Then, the conditions on $q$ reduce to
$\tilde q \in C^{4,1}$ with
\begin{enumerate}
\item[(a)] $\begin{displaystyle}\sup_{x,y} |\partial_2 \tilde q(x,y)| < \infty\end{displaystyle}$, for $\tilde q=b,h,m,v$,
\item[(b)] $\begin{displaystyle}\sup_{x,y} |\partial_1^k \tilde q(x,y)| < \infty\end{displaystyle}$, for $k=0,1,2,3,4$ and $\tilde q=b,h,m$,
\end{enumerate}
and we have
$\partial_A \tilde q^\infty_{A_0}(B)(x) = \partial_2 \tilde q(x,(1,A_0))(1,B)$.
With $\mathtt{x}_t:= (1,\bar A_t)$, the measure $\nu_t : f \mapsto \E[(f,Z_t)]$ satisfies
\begin{multline*}
(f,\nu_t) = (f,\nu_0) + \int_0^t (1,\nu_s) \int \Big(f(0)\partial_2n(x,\mathtt{x}_s) - \partial_2h(x,\mathtt{x}_s)f(x)\Big) \bar A_s(dx) ds \\
+ \int_0^t \int \Big(f'(x) - h(x,\mathtt{x}_s)f(x) + f(0)n(x,\mathtt{x}_s)\Big) \nu_s(dx) ds.
\end{multline*}


\subsection{Density-dependent case}

Suppose that the reproduction parameters are of the form $\hat q((1,\bar A^K))$, where $\hat q:\R_+\to\R$.
We remark that this case can be seen as that given by Ethier and Kurtz \cite{EthKur05}, Chapter 11, Theorem 2.1 and 2.3, with $\beta_l(x) = xb(x)\widecheck p_{x}(l) + xh(x)\widehat p_{x}(l)$, where $\widecheck p_{x}(l)$ and $\widehat p_{x}(l)$ denotes the probability mass functions of $\widecheck\xi_{x}$ and $\widehat\xi_{x}-1$.

Then, the conditions on $q$ further reduce to $\hat q\in C_b^1$
and $\partial_A \hat q^\infty_{A_0}(B)(x) = \hat q'((1,A_0))(1,B)$.
Moreover, the measure $\E[Z_t]$ has a density if $\E[Z_0]$ does.
Indeed, with $\mathtt{x}_t:= (1,\bar A_t)$,
\begin{multline*}
(f,Z_t) = (f,Z_0) +
\int_0^t (-h'(\mathtt{x}_s)(1,Z_s) f + f(0) n'(\mathtt{x}_s)(1,Z_s), \bar{A}_s) ds \\
+ \int_0^t (f' - h(\mathtt{x}_s)f + f(0)n(\mathtt{x}_s), Z_s) ds + \tilde{M}^{f,\infty}_t.
\end{multline*}
Taking $f_\lambda(x) = e^{\lambda x}$ and writing $\tilde M^\lambda_t$ for the martingale, we have
\begin{multline} \label{E:DensityDepCase}
(f_\lambda,Z_t) = (f_\lambda,Z_0) +
\int_0^t \big(n(\mathtt{x}_s)+n'(\mathtt{x}_s)\mathtt{x}_s -h'(\mathtt{x}_s)(f_\lambda,\bar{A}_s) \big) (1,Z_s) ds \\
+ \int_0^t \big(\lambda-h(\mathtt{x}_s)\big) (f_\lambda,Z_s) ds + \tilde M^\lambda_t.
\end{multline}
Taking expectation and letting $\phi(s,\lambda) = n(\mathtt{x}_s)+n'(\mathtt{x}_s)\mathtt{x}_s -h'(\mathtt{x}_s)(f_\lambda,\bar{A}_s)$
and $\psi(s,\lambda) = \lambda-h(\mathtt{x}_s)$,
\begin{equation*}
\E[(f_\lambda,Z_t)] = \E[(f_\lambda,Z_0)] +
\int_0^t \phi(s,\lambda) \E[(1,Z_s)] ds \\
+ \int_0^t \psi(s,\lambda) \E[(f_\lambda,Z_s)] ds.
\end{equation*}
Solving this gives
$$\E[(f_\lambda,Z_t)] = e^{\int_0^t \psi(s,\lambda) ds} \bigg\{ \E[(f_\lambda,Z_0)]
+ \E[(1,Z_0)] \int_0^t \phi(s,\lambda) e^{\int_0^s (\psi(r,0)-\psi(r,\lambda)+\phi(r,0)) dr} ds \bigg\},$$
which reduces to
\begin{multline*}
\E[(f_\lambda,Z_t)]
= e^{\lambda t - \int_0^t h(\mathtt{x}_s) ds} \bigg\{ \E[(f_\lambda,Z_0)]
+ \E[(1,Z_0)] \\
\times \int_0^t \big(n(\mathtt{x}_s)+n'(\mathtt{x}_s)\mathtt{x}_s -h'(\mathtt{x}_s)(f_\lambda,\bar{A}_s)\big)
e^{-\lambda s + \int_0^s (n(\mathtt{x}_r)+(n'(\mathtt{x}_r)-h'(\mathtt{x}_r))\mathtt{x}_r) dr} ds \bigg\}.
\end{multline*}
Inverting the transform, we obtain an expression for the density.
Suppose that $\E[Z_0]$ has density $\mathfrak{z}_0(x)$, then $\E[Z_t]$ has density $\mathfrak{z}_t(x)$ and
\begin{multline*}
\mathfrak{z}_t(x) = e^{- \int_0^t h(\mathtt{x}_s) ds} \bigg\{
\mathfrak{z}_0(x-t) \mathbf{1}_{x>t}
+ \E[(1,Z_0)] \bigg\{ \\
\big(n(\mathtt{x}_{t-x})+n'(\mathtt{x}_{t-x})\mathtt{x}_{t-x}\big) e^{ \int_0^{t-x} (n(\mathtt{x}_r)+(n'(\mathtt{x}_r)-h'(\mathtt{x}_r))\mathtt{x}_r) dr} \mathbf{1}_{x\le t} \\
- \int_{(t-x)\vee0}^t h'(\mathtt{x}_s) e^{\int_0^s (n(\mathtt{x}_r)+(n'(\mathtt{x}_r)-h'(\mathtt{x}_r))\mathtt{x}_r) dr} a(x-t+s,t) ds
\bigg\} \bigg\}
\end{multline*}
where $a(x,t)$ is the density of $\bar A_t$.

In fact, we can solve \eqref{E:DensityDepCase} and obtain
$$(f_\lambda,Z_t) = e^{\int_0^t (\lambda-h(\mathtt{x}_s)) ds} (f_\lambda,Z_0) + \int_0^t e^{\int_s^t (\lambda-h(\mathtt{x}_r)) dr} \Big( \phi(s,\lambda) (1,Z_s) ds + d\tilde M^\lambda_s \Big).$$
Note that
$$(1,Z_t) = e^{\int_0^t \varphi(s)ds} (1,Z_0) + e^{\int_0^t \varphi(r)dr} \int_0^t e^{-\int_0^s \varphi(r)dr} d\tilde M^0_s$$
with
$\varphi(s) = n(\mathtt{x}_s)-h(\mathtt{x}_s) + \big(n'(\mathtt{x}_s)-h'(\mathtt{x}_s)\big) \mathtt{x}_s$.
Thus,
\begin{align*}
&(f_\lambda,Z_t) = e^{\int_0^t (\lambda-h(\mathtt{x}_s)) ds} (f_\lambda,Z_0) \\
&\q + \int_0^t e^{\int_s^t (\lambda-h(\mathtt{x}_r)) dr} \bigg( \phi(s,\lambda) \Big( e^{\int_0^s \varphi(r)dr} (1,Z_0) + e^{\int_0^s \varphi(r)dr} \int_0^s e^{-\int_0^u \varphi(r)dr} d\tilde M^0_u \Big) ds + d\tilde M^\lambda_s \bigg)
\end{align*}
with
$$\big<\tilde{M}^0,\tilde{M}^{\lambda}\big>_t = \int_0^t \Big( w(\mathtt{x}_s)\mathtt{x}_s - h(\mathtt{x}_s)\widehat m(\mathtt{x}_s)\mathtt{x}_s + h(\mathtt{x}_s)\big(1-\widehat m(\mathtt{x}_s)\big)(f_\lambda, \bar A_s) \Big) ds.$$

We can also write the SPDE of $Z$:
$$dZ_t = \Big( \big( n(\mathtt{x}_t)+n'(\mathtt{x}_t)\mathtt{x}_t \big) (1,Z_t)\delta_0 -h(\mathtt{x}_t)Z_t -h'(\mathtt{x}_t)(1,Z_t)\bar A_t -(Z_t)' \Big)dt + d\tilde M^\infty_t$$
with
\begin{align*}
\big<\tilde{M}^{f,\infty}\big>_t 
&= \int_0^t (f^2(0)w(\mathtt{x}_t)+h(\mathtt{x}_t)f^2-2f(0)h(\mathtt{x}_t)\widehat m(\mathtt{x}_t)f, \bar A_s) ds \\
&= f^2(0) w(\mathtt{x}_t) \int_0^t \mathtt{x}_s ds + h(\mathtt{x}_t) \int_0^t (f^2,\bar A_s) ds
- 2f(0)h(\mathtt{x}_t)\widehat m(\mathtt{x}_t) \int_0^t (f,\bar A_s) ds,
\end{align*}
and $\mathtt{x}_t = \mathtt{x}_0 e^{\int_0^t(n(\mathtt{x}_s)-h(\mathtt{x}_s))ds}$.


\subsection{Classical case}

Assume constant parameters $b$, $h$, $m$ and $v$, then, for a test function $f$,
$$(f,Z_t) = (f,Z_0) + \int_0^t (f'-hf+f(0)n,Z_s) ds + \tilde M^{f,\infty}_t.$$
Taking $f_\lambda(x) = e^{\lambda x}$
and writing $\tilde M^\lambda_t$ for the martingale, we have
\begin{equation} \label{E:ClassicalCaseDE}
(f_\lambda,Z_t) = (f_\lambda,Z_0) + (\lambda-h) \int_0^t (f_\lambda,Z_s) ds + n\int_0^t (1,Z_s) ds + \tilde M^\lambda_t
\end{equation}
with
$\big<\tilde{M}^{\lambda}\big>_t = \int_0^t (w+hf_{2\lambda}-2h\widehat mf_\lambda, \bar A_s) ds$.
Taking expectation and solving it, we obtain
$$\E[(f_\lambda,Z_t)] = e^{-ht} \Big(e^{\lambda t}\E[(f_\lambda,Z_0)] + \frac{n}{n-\lambda}\big(e^{nt} - e^{\lambda t}\big)\E[(1,Z_0)]\Big).$$
Suppose that $\E[Z_0]$ has density $\mathfrak{z}_0(x)$, then $\E[Z_t]$ has density $\mathfrak{z}_t(x)$ and
$$\mathfrak{z}_t(x) = e^{-ht}\mathfrak{z}_0(x-t)\mathbf{1}_{x>t}+n\E[(1,Z_0)]e^{(n-h)t}e^{-nx}\mathbf{1}_{x\leq t}.$$

In fact, \eqref{E:ClassicalCaseDE} can also be solved to obtain
$$(f_\lambda,Z_t) = e^{(\lambda-h)t}(f_\lambda,Z_0) + \int_0^t e^{(\lambda-h)(t-s)} \Big( n(1,Z_s) ds + d\tilde M^\lambda_s \Big).$$
With
$(1,Z_t) = e^{(n-h)t} (1,Z_0) + \int_0^t e^{(n-h)(t-s)} d\tilde M^0_s$,
we can write
\begin{align*}
(f_\lambda,Z_t) &= e^{(\lambda-h)t}(f_\lambda,Z_0) + \int_0^t e^{(\lambda-h)(t-s)} \bigg( ne^{(n-h)s} \Big( (1,Z_0) + \int_0^s e^{-r(n-h)} d\tilde M^0_r\Big) ds + d\tilde M^\lambda_s \bigg) \\
&= e^{(\lambda-h)t}(f_\lambda,Z_0) + \frac{n}{n-\lambda} (1,Z_0) \big(e^{(n-h)t}-e^{(\lambda-h)t}\big) \\
&\q\q+ n e^{(\lambda-h)t} \int_0^t e^{(n-\lambda)s} \int_0^s e^{-(n-h)r} d\tilde M^0_r ds + \int_0^t e^{(\lambda-h)(t-s)} d\tilde M^\lambda_s,
\end{align*}
where
$$\big<\tilde{M}^0,\tilde{M}^{\lambda}\big>_t = \int_0^t \big(w-h\widehat m + h(1-\widehat m)f_\lambda, \bar A_s \big) ds.$$

The SPDE of $Z$ is
$$dZ_t(dx) = \big( n(1,Z_t)\delta_0(dx) -hZ_t(dx) -(Z_t)'(dx) \big)dt + d\tilde M^\infty_t(dx)$$
with
$$\big<\tilde{M}^{f,\infty}\big>_t
= f^2(0) \frac{w}{n-h} (1,\bar A_0) \big(e^{(n-h)t}-1\big)
+ h \int_0^t (f^2,\bar A_s) ds
- 2f(0)h\widehat m \int_0^t (f,\bar A_s) ds.$$
In the case where the density exists,
\begin{multline*}
\big<\tilde{M}^{f,\infty}\big>_t
= f^2(0) \frac{w}{n-h} \big(e^{(n-h)t}-1\big) \int a_0(x)dx \\
+ h \int_0^t \int f^2(x)a(x,s)dx ds
- 2f(0)h\widehat m \int_0^t \int f(x)a(x,s)dx ds,
\end{multline*}
where
$$a(x,t) = \begin{cases} a(0,t-x)e^{-hx}, x\le t \\ a_0(x-t)e^{-ht}, x>t \end{cases}$$
with $a(0,t)=n\int a(x,t)dx$ and $a(x,0)=a_0(x)$.
In particular,
$$\big<\tilde{M}^{1,\infty}\big>_t = \frac{w+h-2h\widehat m}{n-h} (1,\bar A_0) (e^{(n-h)t}-1).$$


\section*{Acknowledgements}
This research was supported by the Australian Research Council Grant DP150103588.
The authors are grateful to the anonymous referees for their valuable comments that led to an improvement of the paper.

\end{document}